\RequirePackage{fix-cm}
\documentclass[smallcondensed]{svjour3} 
\usepackage[numbers]{natbib}
\usepackage{graphicx}
\usepackage{amsmath}
\usepackage{mathrsfs}
\usepackage{indentfirst} 
\usepackage{enumerate}
\usepackage{geometry}
\usepackage{tikz}
\usetikzlibrary{arrows.meta}
\usepackage{subfigure}
\usepackage{amssymb}
\usepackage[misc]{ifsym}
\usepackage{mathtools}
\usepackage{mathtools}
\usepackage{soul}
\setstcolor{blue}
\usepackage[pagewise]{lineno}

\smartqed 
\begin{document}
\definecolor{applegreen}{rgb}{0.55, 0.71, 0.0}
\title{Travelling wave solutions in a negative nonlinear diffusion-reaction model}

\author{Yifei Li$^1$\and
        Peter van Heijster$^1$\and
         Robert Marangell$^2$\and
         Matthew J. Simpson$^1$  
}
\institute{
\Letter \quad Peter van Heijster\\
\phantom{\Letter}\quad 
petrus.vanheijster@qut.edu.au\\
\\$^1$ \quad School of Mathematical Sciences, Queensland University of   Technology, Brisbane, Australia\\ \\
       $^2$ \quad School of Mathematics and Statistics, University of Sydney, Sydney, Australia
}

\date{Received: date / Accepted: date}
\maketitle
\begin{abstract}
We use a geometric approach to prove the existence of smooth travelling wave solutions of a nonlinear diffusion-reaction equation with logistic kinetics and a convex nonlinear diffusivity function which changes sign twice in our domain of interest. We determine the minimum wave speed, $c^*$, and investigate its relation to the spectral stability of the travelling wave solutions.
\end{abstract}

\section{Introduction}
Invasion processes have been studied with mathematical models, especially partial differential equations (PDEs), for many years; see, for example, \citet{murray2002mathematical} and references therein. These models describe, for instance, how cells are transported to new areas in which they persist, proliferate, and spread \citep{mack2000biotic}. To incorporate information about individual-level behaviours in invasion processes, lattice-based discrete models are widely used \citep{deroulers2009modeling,johnston2017co,johnston2012mean,simpson2010migration}. In these discrete models, individual agents are permitted to move, proliferate and die on a lattice, and the average density of agents is related to PDE descriptions obtained using truncated Taylor series in the continuum limit \citep{anguige2009one,codling2008random}. The macroscopic behaviour described by the PDEs in terms of expected agent density reflects the individual microscopic behaviour. Travelling wave solutions are of particular interest among the macroscopic behaviours arising from these continuum models, as they reflect various modes of microscopic invasive behaviours. One famous model exhibiting travelling wave solutions is the Fisher-KPP equation (KPP refers to Kolmogorov, Petrovsky, Piskunov) proposed in 1937 to study population dynamics with linear diffusion and logistic growth \citep{fisher1937wave,ararticle}. The existence and stability of travelling wave solutions of the Fisher-KPP equation has been widely studied, see, for instance, \citet{aronson1978multidimensional,fisher1937wave,harley2015numerical,ararticle,larson1978transient,murray2002mathematical} and \citet{sherratt1998transition}.

The Fisher-KPP equation can be derived as a continuum limit of a discrete model under the assumption that the population of cells can be treated as a uniform population without any differences in subpopulations \citep{bramson1986microscopic}. However, differences between individual and collective behaviour have been observed in cell biology and ecology in practice. For instance, in cell biology, isolated cells called \emph{leader cells} are more motile than the grouped cells, called \emph{follower cells} \citep{poujade2007collective}. Also, contact interactions lead to different motility rates between isolated cells and grouped cells in the migration of breast cancer cells \citep{simpson2010migration,simpson2014pioneer}, glioma cells \citep{khain2011collective}, would healing processes \citep{Khain2007} and the development of the enteric nervous system \citep{druckenbrod2007behavior}. In ecology, the population growth rate of some species decreases as their populations reach small sizes or low densities \citep{courchamp1999inverse}. This phenomenon is usually referred to as the Allee effect \citep{allee1932studies}. 

To describe the invasion process and reflect the difference between collective and individual behaviour, Johnston and coworkers introduced a discrete model considering birth, death and movement events of agents that are isolated or grouped on a simple one-dimensional lattice \citep{johnston2017co}. A discrete conservation statement describing $\delta U_{j}$, which is the change of the occupancy of a lattice site $j$ during a time step $\tau$, gives
\begin{equation}
\begin{aligned}
    \delta U_j=&\frac{P^i_m}{2}[U_{j-1}(1-U_j)(1-U_{j-2})+U_{j+1}(1-U_j)(1-U_{j+2})\\
    &-2U_{j}(1-U_{j-1})(1-U_{j+1})]\\
    &+\frac{P^g_m}{2}[U_{j-1}(1-U_j)+U_{j+1}(1-U_j)-U_{j}(1-U_{j-1})-U_{j}(1-U_{j+1})]\\
    &-\frac{P^g_m}{2}[U_{j-1}(1-U_j)(1-U_{j-2})+U_{j+1}(1-U_j)(1-U_{j+2})\\
    &-2U_{j}(1-U_{j-1})(1-U_{j+1})]\\
    &+\frac{P^i_p}{2}[U_{j-1}(1-U_j)(1-U_{j-2})+U_{j+1}(1-U_j)(1-U_{j+2})]\\
    &+\frac{P^g_p}{2}[U_{j-1}(1-U_j)+U_{j+1}(1-U_j)]\\
    &-\frac{P^g_p}{2}[U_{j-1}(1-U_j)(1-U_{j-2})+U_{j+1}(1-U_j)(1-U_{j+2})]\\
    &-P^i_d[U_j(1-U_{j-1})(1-U_{j+1})]-P^g_dU_j+P^g_d[U_j(1-U_{j-1})(1-U_{j+1})].
\end{aligned}
\label{1}
\end{equation}
Here, $U_j$ represents the probability that an agent occupies lattice site $j$, thus, $1-U_j$ represents the probability that lattice site $j$ is vacant \citep{simpson2010cell}. $P^i_m$ and $P^g_m$ represents the probability per time step that isolated or grouped agents, respectively, attempt to step to a nearest neighbour lattice site; $P^i_p$ and $P^g_p$ represents the probability per time step that isolated or grouped agents, respectively, attempt to undergo a proliferation event and deposit a daughter agent at a nearest neighbour lattice site; $P_d^i$ and $P_d^g$ represents the probability per time step that isolated or grouped agents, respectively, die, and are removed from the lattice. See Figure \ref{D(u)andR(u)}a for a schematic of the lattice-based discrete model.

To obtain a continuous description, Johnston and coworkers treat $U_j$ as a continuous function, $U(x,t)$, and divide (\ref{1}) by the time step $\tau$. Next, they expanded all terms in (\ref{1}) in a Taylor series around $x=j\Delta$, where $\Delta$ is the lattice spacing, and neglect terms of $\mathcal{O}(\Delta^3)$ \citep{simpson2010cell}. As $\Delta\to0$ and $\tau\to0$ with the ratio ${\Delta^2}/{\tau}$ held constant \citep{codling2008random,simpson2010cell}, they obtained a nonlinear diffusion-reaction equation
\begin{equation}
\label{RDE_1}
    \begin{aligned}
    \frac{\partial U}{\partial t}=\frac{\partial}{\partial x}\left(D(U)\frac{\partial U}{\partial x}\right)+{R}\left(U\right),
    \end{aligned}
\end{equation} 
where 
\begin{equation}
\label{D(u)2}
D\left(U\right)=D_i\left(1-4U+3U^2\right)+D_g\left(4U-3U^2\right),
\end{equation}
is the nonlinear diffusivity function, and
\begin{equation}
\label{R(u)1}
R\left(U\right)=\lambda_g U\left(1-U\right)+\left(\lambda_i-\lambda_g-K_i+K_g\right)U\left(1-U\right)^2-K_g U,
\end{equation}
is the kinetic term. Furthermore, the parameters are given by
\begin{equation}
    \nonumber
    \begin{aligned}
    &D_g=\lim_{\Delta,\tau\to0}\frac{P^g_m\Delta^2}{2\tau},\quad
    &&D_i=\lim_{\Delta,\tau\to0}\frac{P^i_m\Delta^2}{2\tau},\quad
    &&&\lambda_g=\lim_{\tau\to0}\frac{P^g_p}{\tau},\\
    &\lambda_i=\lim_{\tau\to0}\frac{P^i_p}{\tau},\quad
    &&K_g=\lim_{\tau\to0}\frac{P^g_d}{\tau},\quad
    &&&K_i=\lim_{\tau\to0}\frac{P^i_d}{\tau},
    \end{aligned}
\end{equation}
where we require that $P^i_p,P^g_p,P^i_d,P^g_d$ are $\mathcal{O}(\tau)$ \citep{simpson2010cell}. Here, $U(x,t)$ denotes the total density of the agents at position $x\in\mathbb{R}$ and time $t\in\mathbb{R}_{+}$; $D_i\ge0$ and $D_g\ge0$ are diffusivities of the isolated and grouped agents, respectively; $\lambda_i\ge0$ and $\lambda_g\ge0$ are the proliferation rates of isolated and grouped agents, respectively; $K_i\ge0$ and $K_g\ge0$ are the death rates of isolated and grouped agents, respectively \citep{johnston2017co}.
\begin{figure}
\centering
\begin{tikzpicture}[x=4cm,y=3cm, scale=1]
\draw[thick] (0,0.9) -- (1.6,0.9);
\draw[thick] (0,1.4) -- (1.6,1.4);
\draw[thick] (0,0.9) -- (0,1.4);
\draw[thick] (1.6,0.9) -- (1.6,1.4);
\foreach \pos in {0.2,0.4,0.6,0.8,1,1.2,1.4}
\draw[thick] (\pos,0.9) -- (\pos,1.4);
\node [black,right] at (0,1.35) {$1$};
\node [black,right] at (0.2,1.35) {$2$};
\node [black,right] at (0.4,1.35) {$3$};
\node [black,right] at (0.6,1.35) {$4$};
\node [black,right] at (0.8,1.35) {$5$};
\node [black,right] at (1,1.35) {$6$};
\node [black,right] at (1.2,1.35) {$7$};
\node [black,right] at (1.4,1.35) {$8$};
\foreach \pos in {0.2,0.4,0.6,0.8,1,1.2,1.4}
\fill[pink] (0.3,1.15) circle (10pt);
\node [black] at (0.3,1.15) {$A$};
\fill[cyan] (0.9,1.15) circle (10pt);
\node [black] at (0.9,1.15) {$B$};
\fill[cyan] (1.1,1.15) circle (10pt);
\node [black] at (1.1,1.15) {$C$};
\fill[cyan] (1.3,1.15) circle (10pt);
\node [black] at (1.3,1.15) {$D$};
\draw[-latex] (0.25,1.5) -- (0.1,1.5);
\draw[-latex] (0.35,1.5) -- (0.5,1.5);
\draw[-latex] (0.85,1.5) -- (0.7,1.5);
\draw[-latex] (1.35,1.5) -- (1.5,1.5);
\node [black,above] at (0.13,1.5) {$\dfrac{P^i_m}{2}$};
\node [black,above] at (0.47,1.5) {$\dfrac{P^i_m}{2}$};
\node [black,above] at (0.73,1.5) {$\dfrac{P^g_m}{2}$};
\node [black,above] at (1.47,1.5) {$\dfrac{P^g_m}{2}$};
\draw[-latex] (0.1,0.8) -- (0.1,1);
\node [black,below] at (0.1,0.8) {$\dfrac{P^i_p}{2}$};
\draw[-latex] (0.5,0.8) -- (0.5,1);
\node [black,below] at (0.5,0.8) {$\dfrac{P^i_p}{2}$};
\draw[-latex] (0.7,0.8) -- (0.7,1);
\node [black,below] at (0.7,0.8) {$\dfrac{P^g_p}{2}$};
\draw[-latex] (1.5,0.8) -- (1.5,1);
\node [black,below] at (1.5,0.8) {$\dfrac{P^g_p}{2}$};
\draw[-latex] (0.3,1) -- (0.3,0.8);
\node [black,below] at (0.3,0.8) {$P^i_d$};
\draw[-latex] (0.9,1) -- (0.9,0.8);
\node [black,below] at (0.9,0.8) {$P^g_d$};
\draw[-latex] (1.1,1) -- (1.1,0.8);
\node [black,below] at (1.1,0.8) {$P^g_d$};
\draw[-latex] (1.3,1) -- (1.3,0.8);
\node [black,below] at (1.3,0.8) {$P^g_d$};
\draw[-latex, thick] (-0.1,1.2) -- (-0.1,0);
\node [black,left] at (-0.1,0.6) {$t+\tau$};
\draw[thick] (0,-0.2) -- (1.6,-0.2);
\draw[thick] (0,0.3) -- (1.6,0.3);
\draw[thick] (0,-0.2) -- (0,0.3);
\draw[thick] (1.6,-0.2) -- (1.6,0.3);
\foreach \pos in {0.2,0.4,0.6,0.8,1,1.2,1.4}
\draw[thick] (\pos,-0.2) -- (\pos,0.3);
\node [black,right] at (0,0.25) {$1$};
\node [black,right] at (0.2,0.25) {$2$};
\node [black,right] at (0.4,0.25) {$3$};
\node [black,right] at (0.6,0.25) {$4$};
\node [black,right] at (0.8,0.25) {$5$};
\node [black,right] at (1,0.25) {$6$};
\node [black,right] at (1.2,0.25) {$7$};
\node [black,right] at (1.4,0.25) {$8$};
\fill[pink] (0.3,0.05) circle (10pt);
\fill[pink] (0.7,0.05) circle (10pt);
\fill[cyan] (1.1,0.05) circle (10pt);
\fill[cyan] (1.3,0.05) circle (10pt);
\fill[cyan] (1.5,0.05) circle (10pt);
\node [black] at (0.3,0.05) {$A$};
\node [black] at (0.7,0.05) {$B$};
\node [black] at (1.1,0.05) {$C$};
\node [black] at (1.3,0.05) {$D$};
\node [black] at (1.5,0.05) {$E$};
\node [black,below] at (0.75,-0.2) {(a)};
\end{tikzpicture}
\hfill
\begin{tikzpicture}[x=6cm,y=3.5cm, scale=1]
\draw[thick] (0,-0.2) -- (0,1.5);
\draw[thick] (0,0) -- (1,0) 
node[right,very thick] {$U$} coordinate( x axis);
\node [black,above] at (0.83,0) {$\beta$};
\node [black,above] at (0.5,0) {$\alpha$};
\node [black,above] at (0.68,0) {$\frac{2}{3}$};
\node [black,right] at (0.4,1.35) {$R(U)$};
\node [black,right] at (0.25,0.4) {$D(U)$};
\draw [color=cyan,very thick] [label] plot file {D3.dat};
\draw [color=orange,very thick] [label] plot file {R1.dat};
\foreach \pos in {1}  
\draw[shift={(\pos,0)}] (0,0.01) -- (0,-0.01) node[below] {$\pos$};
\foreach \pos in {0}  
\draw[shift={(0,\pos)}] (0.01,0) -- (-0.01,0) node[left] {$\pos$};
\fill[black] (0.668,0) circle (2pt);
\fill[black] (0.5,0) circle (2pt);
\fill[black] (0.83,0) circle (2pt);
\node [black,below] at (0.5,-0.2) {(b)};
\end{tikzpicture}
\caption{(a) describes one possible time step of the lattice-based discrete model of \citet{johnston2017co}: a new grouped agent (agent E) is born and the grouped agent B moves from lattice site 5 to lattice site 4 to become an isolated agent. Pink circles represent isolated agents with birth rate $P^i_p$, death rate $P^i_d$ and motility rate related to $P^i_m$; cyan circles represent grouped agents with birth rate $P^g_p$, death rate $P^g_d$ and motility rate $P^g_m$. (b) presents a diffusivity function $D(U)$, given by (\ref{D(u)2}) (cyan curve) satisfying $D_i>4D_g$ which makes $D(U)$ change sign twice on $(0,1)$, and the kinetic term $R(U)$, given by (\ref{R(u)2}) (orange curve) which is positive on $(0,1)$ and zero at end points $U=0$ and $U=1$.}
\label{D(u)andR(u)}
\end{figure}
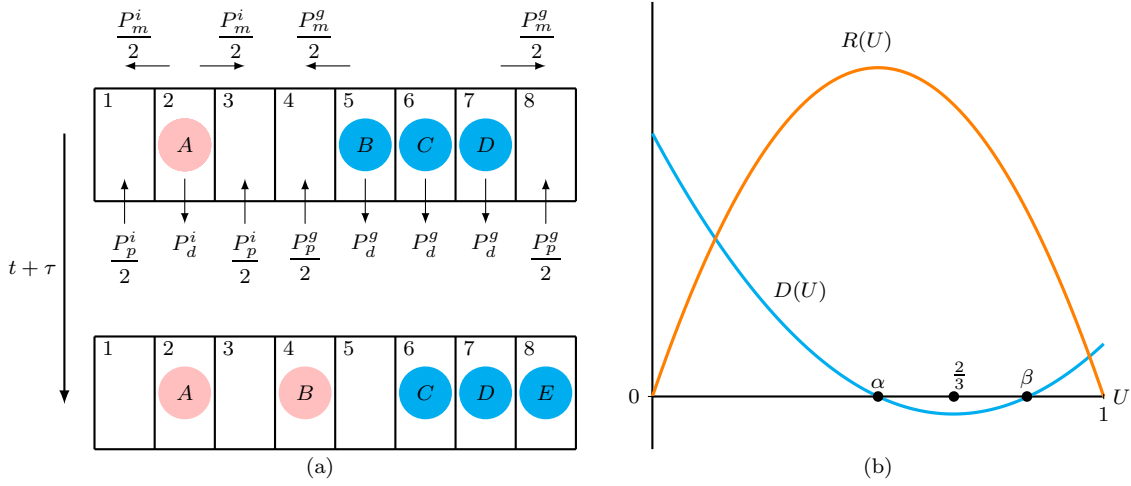
Note that this particular form (\ref{RDE_1}) 
was proposed by \citet{johnston2017co}.  
This was one of the first studies that proposed a nonlinear diffusion-reaction model to a mean-field description of a lattice-based stochastic model incorporating agent movement, proliferation and death.
%
Previous work leading to nonlinear diffusion equations only considered the movement of agents and thus did not involve kinetic terms \citep{johnston2012mean,anguige2009one}.

In this manuscript, we study the effect that aggregation, which is modelled with a nonlinear diffusivity function that goes negative \citep{simpson2010model}, has on the dynamics of the continuous PDE model. Therefore, we assume that $D_i>4D_g$ such that $D(U)$ given by (\ref{D(u)2}) is convex and changes sign twice in our domain of interest (additionally, see Section~\ref{SS:POS} for a short discussion related to the other case). For simplicity, we furthermore assume equal proliferation rates, $\lambda=\lambda_i=\lambda_g$, and no agent death, $K_i=K_g=0$. This way, the kinetic term simplifies to a logistic term
\begin{equation}
\label{R(u)2}
    R\left(U\right)=\lambda U\left(1-U\right),
\end{equation}
and $D\left(U\right)$ has a sign condition:
\begin{equation}
\label{signcondition1}
    D\left(U\right)>0 \quad \text{for}\quad U\in\left[0,\alpha\right)\cup\left(\beta,1\right],\quad D\left(U\right)<0\quad \text{for} \quad U\in\left(\alpha,\beta\right),
\end{equation}
where the interval where $D(U)<0$ is centred at $U=2/3$, and $\alpha, \beta$ are given by
\begin{equation} 
\label{X}
    \alpha=\frac{2}{3}-\frac{\sqrt{D_{i}^{2}+4D_{g}^2-5D_{i}D_{g}}}{3\left(D_i-D_g\right)},\quad\beta=\frac{2}{3}+\frac{\sqrt{D_{i}^{2}+4D_{g}^2-5D_{i}D_{g}}}{3\left(D_i-D_g\right)},
\end{equation}
with $1/3<\alpha<2/3$ and $2/3<\beta<1$, see Figure \ref{D(u)andR(u)}b. That is, we have negative diffusion for $U\in(\alpha,\beta)$. The relation that $D_i$ is larger than $D_g$ indicates that isolated agents are more active than grouped agents, which agrees with the experimental observation that \emph{leader cells} are more motile than \emph{follower cells} \citep{poujade2007collective,simpson2014pioneer}. 
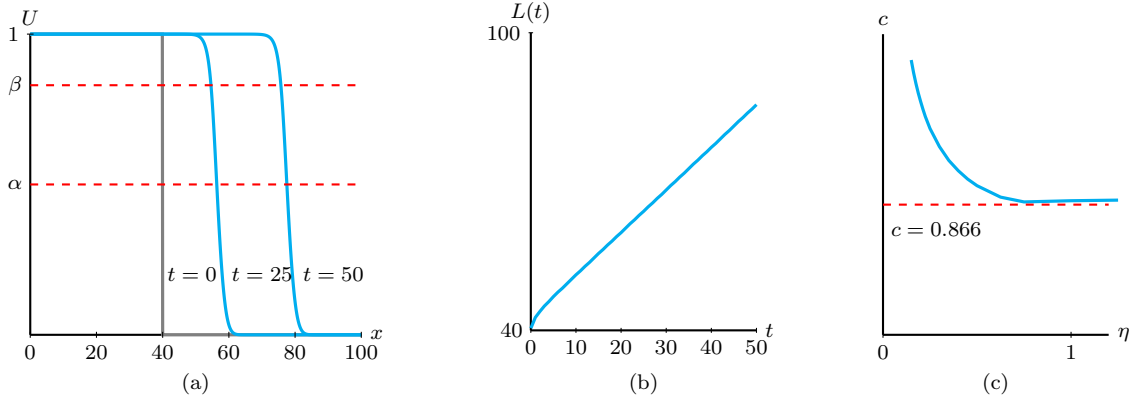
\begin{figure}
\centering
\begin{tikzpicture}[x=2.2cm,y=200cm, scale=0.02]
\draw[thick] (0,0) -- (0,1) node[above,very thick] {$U$} coordinate(y axis);
\draw[thick] (0,0) -- (100,0) 
node[right,very thick] {$x$} coordinate( x axis);
\draw [color=gray,very thick] [label] plot file {travellingwavesolution_1.dat};
\draw [color=cyan,very thick] [label] plot file {travellingwavesolution_2.dat};
\draw [color=cyan,very thick] [label] plot file {travellingwavesolution_3.dat};
\foreach \pos in {0,20,40,60,80,100}
\draw[shift={(\pos,0)}] (0,0.01) -- (0,-0.01) node[below] {$\pos$};
\foreach \pos in {1}  
\draw[shift={(0,\pos)}] (1,0) -- (-1,0) node[left] {$\pos$};
\node [black,right] at (39,0.2) {$t=0$};
\node [black,right] at (58.5,0.2) {$t=25$};
\node [black,right] at (80,0.2) {$t=50$};
\draw[thick,dashed,red] (0,0.83) -- (100,0.83);
\draw[thick,dashed,red] (0,0.5) -- (100,0.5);
\node [black,left] at (0,0.83) {$\beta$};
\node [black,left] at (0,0.5) {$\alpha$};
\node [black,below] at (50,-0.1) {(a)};
\end{tikzpicture}
\hfill
\begin{tikzpicture}[x=2cm,y=2.2cm, scale=0.03]
\draw[thick] (0,0) -- (0,60) node[above,very thick] {$L(t)$} coordinate(y axis);
\draw[thick] (0,0) -- (50,0) 
node[right,very thick] {$t$} coordinate( x axis);
\draw [color=cyan,very thick] [label] plot file {travellingwavesolution_speed.dat};
\foreach \pos in {0,10,20,30,40,50}  
\draw[shift={(\pos,0)}] (0,0.5) -- (0,-0.5) node[below] {$\pos$};
\draw[shift={(0,0)}] node[left] {$40$};
\draw[shift={(0,60)}] (0.5,0) -- (-0.5,0) node[left] {$100$};
\node [black,below] at (25,-7) {(b)};
\end{tikzpicture}
\hfill
\begin{tikzpicture}[x=2.5cm,y=2cm, scale=1]
\draw[thick] (0,0) -- (0,2) node[above,very thick] {$c$} coordinate(y axis);
\draw[thick] (0,0) -- (1.2,0) 
node[right,very thick] {$\eta$} coordinate( x axis);
\foreach \pos in {0,1}  
\draw[shift={(\pos,0)}] (0,0.02) -- (0,-0.02) node[below] {$\pos$};
\draw [color=cyan,very thick] [label] plot file {speedwithICs.dat};
\draw[thick,dashed,red] (0,0.866) -- (1.2,0.866);
\node [black,below] at (0.6,-0.2) {(c)};
\node [black,right] at (0,0.7) {$c=0.866$};
\end{tikzpicture}
\caption{(a) shows the evolution of a Heaviside initial condition to a smooth travelling wave solution obtained by simulating (\ref{RDE_1}) with (\ref{D(u)2}) and (\ref{R(u)2}) with parameters $D_i=0.25$, $D_g=0.05$ and $\lambda=0.75$. We use a finite difference method with space step $\delta x=0.1$, time step $\delta t=0.01$ and no-flux boundary conditions. Notice that $D(U)=0$ at $\alpha=0.5$ and $\beta\approx0.83$. (b) measures the position of the wave $L(t)$ by looking for the left-most leading edge point where $U$ is smaller than $10^{-5}$, indicating that the solution is travelling at a constant speed $c=0.864$. (c) gives the wave speed as a function of the initial condition $U(x,0)=1/2+\text{tanh}\left(-\eta (x-40)\right)/2$. Notice that as $\eta$ grows to infinity this initial condition limits to the Heaviside initial condition used for the simulation in (a), and the wave speed converges to $c\approx0.864$. The minimum wave speed $c^*=2\sqrt{\lambda D_i}\approx0.866$ (\ref{minimumwave}).}
\label{travellingwavesolutionpicture1}
\end{figure}
\citet{ferracuti2009travelling} showed the existence of travelling wave solutions for a range of positive wave speeds for (\ref{RDE_1}) with general convex $D(U)$ that changes sign twice on $(0,1)$ and $R(U)$ given by $(\ref{R(u)2})$ based on the \emph{comparison method} introduced by \citet{aronson1978multidimensional}. Related studies proved the existence of travelling wave solutions for a similar range of speeds for nonlinear diffusion-reaction equations with different $D(U)$ and different $R(U)$: \citet{malaguti2003sharp} studied (\ref{RDE_1}) with a logistic kinetic term and a nonlinear diffusivity function satisfying
\begin{equation}
\nonumber
    D(0)=0\quad\text{and}\quad D(0)>0\quad \text{for all}\quad U\in(0,1].
\end{equation}
\citet{maini2006diffusion} studied (\ref{RDE_1}) with a logistic kinetic term and a nonlinear diffusivity function satisfying
\begin{equation}
    \label{maini1}
    D(U)>0\quad\text{in}\quad(0,\theta)\quad \text{and}\quad D(U)<0 \quad \text{in}\quad U\in(\theta,1),
\end{equation}
for some given $\theta\in(0,1)$ and with $D(0)=D(\theta)=D(1)=0$. In addition, \cite{maini2007aggregative} studied (\ref{RDE_1}) with (\ref{maini1}) and a bistable kinetic term satisfying
\begin{equation}
\nonumber
    R(0)=R(\phi)=R(1)=0,\quad R(U)<0\quad \text{in} \quad U\in(0,\phi)\quad \text{and}\quad R(U)>0\quad \text{in} \quad U\in(\phi,1).
\end{equation}

A travelling wave solution of {(\ref{RDE_1})} is a solution that travels with constant speed $c>0$ and constant wave shape, and that asymptotes to $1$ as $x\to-\infty$ and to $0$ as $x\to\infty$ (i.e. the roots of $R(U)$).
We only consider positive wave speeds since (\ref{RDE_1}) with $(\ref{D(u)2})$ and $(\ref{R(u)2})$ is monostable with a Fisher-KPP imprint, that is, $U\equiv1$ is a PDE stable solution of (\ref{RDE_1}), while $U\equiv0$ is a PDE unstable solution (in an appropriate function space which will be introduced in Section~\ref{S:ST}). Hence, to study travelling wave solutions we introduce the travelling wave coordinate $z=x-ct$, where $z\in\mathbb{R}$ and $c>0$, and write (\ref{RDE_1}) in its travelling wave coordinate
\begin{equation}
\label{stabilitysec2}
    \frac{\partial U}{\partial t}=\frac{\partial}{\partial z}\left(D(U)\frac{\partial U}{\partial z}\right)+c\frac{\partial U}{\partial z}+R(U).
\end{equation}
A travelling wave solution is now a stationary solution to (\ref{stabilitysec2}), that is, ${\partial U}/{\partial t}=0$ \citep{sandstede2002stability}. In other words, a travelling wave solution is a solution to the second-order ordinary differential equation (ODE)
\begin{equation}
\label{ODE_1}
    \frac{d}{d z}\left(D(u)\frac{du}{dz}\right)+c\frac{du}{dz}+R(u)=0,
\end{equation}
with asymptotic boundary conditions $\displaystyle{\lim_{z\to-\infty}}u=1$ and $\displaystyle{\lim_{z\to\infty}}u=0$.

In this manuscript, we show the following result:
\begin{theorem} \label{THEOREM}
Model (\ref{RDE_1}) with (\ref{D(u)2}) and (\ref{R(u)2}) and $D_i>4D_g$ supports smooth monotone nonnegative travelling wave solutions for
\begin{equation}
    \label{minimumwave}
    c\ge2\sqrt{\lambda D_i}=:c^*.
\end{equation}
\end{theorem}

This theorem agrees with the result of \citet{ferracuti2009travelling}, and because of the specific nonlinear diffusivity function, we can further extend their results. Moreover, instead of the comparison method used by \citet{ferracuti2009travelling}, we use a geometric approach to prove the existence of travelling wave solutions. This geometric approach has the advantage that it can also be used to study shock-fronted, discontinuous travelling wave solutions \citep{wechselberger2010folds,harley2014novel,harley2014existence}. While shock-fronted travelling wave solutions are not the focus in this manuscript, we show in the final section that they do exist for (\ref{R(u)2}) with different $D(U)$, see Figure \ref{phaseplane_pic4}a in Section~\ref{SS:shocks}. The lower bound $c^*$ in Theorem~\ref{THEOREM} is often called the \emph{minimum wave speed} as it represents the monotone nonnegative travelling wave solutions with the lowest wave speed \citep{murray2002mathematical}. Numerical simulations show that (\ref{RDE_1}) with (\ref{D(u)2}) and (\ref{R(u)2}) indeed support smooth travelling wave solutions even though the nonlinear diffusivity function goes negative. Moreover, the speed relates to the initial condition, and the wave speed converges to the minimum wave speed $c^*$ as the initial condition limits to the Heaviside initial condition, see Figure \ref{travellingwavesolutionpicture1}. We will also show the connection between the existence of smooth monotone nonnegative travelling wave solutions, the spectrum of the travelling wave solutions, and the minimum wave speed $c^*$.

This manuscript is organised as follows. We prove Theorem~\ref{THEOREM} in Section~\ref{S:EXIST} by using desingularisation techniques \citep{aronson1980density} and detailed phase plane analysis which have not been applied to (\ref{RDE_1}) before. In Section~\ref{S:ST}, we determine the spectral properties of the travelling wave solutions and show how the minimum wave speed $c^*$ is related to absolute instabilities \citep{sandstede2002stability,kapitula2013spectral,sherratt2014mathematical}. Some interesting results for different nonlinear diffusivity functions with the same kinetic term $(\ref{R(u)2})$ are discussed in Section~\ref{S:SUM}. Here, we also discuss the implications of the analytical results for the discrete model. Note that throughout the manuscript all theoretical results are supported by high-quality numerical simulations of the continuum PDE model.

\begin{remark}
Many essential mathematical questions related to, for instance, well-posedness, remain open for PDEs with forward-backward diffusion, i.e. models like \eqref{RDE_1} with nonlinear diffusivity functions that change sign. For instance, 
the well-studied Perona-Malik model \citep{perona1990scale} from image analysis with forward-backward diffusion, but without a kinetic term, is ill-posed \citep{weickert1998anisotropic}. See also \cite{hollig1983existence}.

The ill-posedness of these PDEs with forward-backward diffusion can often be addressed by adding a small regularisation term, like a viscous regularisation term \citep{novick1991stable} or a nonlocal Cahn-Hilliard-type regularisation term \citep{pego1989front}. For the Perona-Malik model this was done, with another type of regularisation term, by \cite{barenblatt1993degenerate}. Interestingly, different regularisations can have different singular limits, in particular, when shock solutions are formed (see also Section~\ref{SS:shocks}). This is particularly interesting when you realise that most numerical schemes introduce some artificial regularisation. In other words, different numerical schemes can correctly yield different solutions \citep{witelski1995shocks}. Also, recall that in the derivation of the continuum limit higher order terms were ignored. These higher order terms potentially have a regularising effect and can shed light on the ``right'' type of regularisation. 

Since we are constructing smooth solutions in this manuscript, we do not address the question of well-posedness of \eqref{RDE_1}.
\end{remark}

\section{Existence of travelling wave solutions} \label{S:EXIST}
\subsection{Transformation and Desingularisation}
We use a dynamical systems approach to analyse the second-order ODE (\ref{ODE_1}) whose solutions that asymptote to $\displaystyle{\lim_{z\to-\infty}}u=1$ and $\displaystyle{\lim_{z\to\infty}}u=0$ correspond to travelling wave solutions of \eqref{RDE_1}. Upon introducing $p:=D(u)du/dz$, \eqref{ODE_1} can be written as a singular system of first-order ODEs
\begin{equation}
\label{ODEsystem_desingularised_0}
    \left\{\begin{aligned}
     D(u)\frac{du}{dz}&=p,\\
     D(u)\frac{dp}{dz}&=-cp-D(u)R(u).
\end{aligned}\right.
\end{equation}
Travelling wave solutions of (\ref{RDE_1}) now correspond to heteroclinic orbits of (\ref{ODEsystem_desingularised_0}) connecting $(1,0)$ to $(0,0)$. 
Note that $p>0$ if $du/dz<0$ and $D(u)<0$. Thus, while we expect that the derivative of a travelling wave solution is always negative, $p$ is not necessarily always negative. The nullclines of system (\ref{ODEsystem_desingularised_0}) are given by $p=0$ and $-cp-D(u)R(u)=0$ with the constraint that $D(u)\ne0$. 
However, $D(u)$ vanishes when $u=\alpha$ and $u=\beta$~\eqref{X}, and
system (\ref{ODEsystem_desingularised_0}) is thus undefined, or singular, along the lines $u=\alpha$ and $u=\beta$ \citep{simpson2007nonmonotone}. These lines are
sometimes called the \emph{walls of singularities} \citep{pettet2000lotka,wechselberger2010folds,harley2014existence}.
Trajectories can potentially still cross through these walls at special points, sometimes referred to as \emph{holes in the wall} \citep{pettet2000lotka,wechselberger2010folds,harley2014existence}, when, in addition to $D(u)=0$, the right hand sides of the singular system also vanish (and if the holes in the wall are of the correct type \citep{wechselberger2005existence,wechselberger2010folds,harley2014existence}). 
These holes in the wall, and the trajectories crossing them, can often be linked to {\emph{folded singularities}} and {\emph{canard solutions}} upon embedding the singular system into higher-dimensional singularly perturbed systems with {\emph{folded critical manifolds}}, we refer to \citet{szmolyan2001canards,wechselberger2005existence,wechselberger2010folds,harley2014existence}, and references therein, for more details on this now well-established theory.
For system (\ref{ODEsystem_desingularised_0}) the holes in the wall are $(\alpha,0)$ and $(\beta,0)$.
%
%
%
%
To remove the singularities, we desingularise system (\ref{ODEsystem_desingularised_0}) by introducing a stretched variable $\xi$ satisfying $D(u)d\xi=dz$ \citep{aronson1980density,murray2002mathematical,sanchez1994existence,harley2014existence}. Subsequently, system
(\ref{ODEsystem_desingularised_0}) becomes 
\begin{equation}
\label{ODEsystem_desingularised_1}
    \left\{\begin{aligned}
    &\frac{du}{d\xi}=p,\\
    &\frac{dp}{d\xi}=-cp-D(u)R(u).
\end{aligned}\right.
\end{equation}
Here we see that the desingularisation changes the independent variable $z$ in a nonlinear fashion, but it does not change the dependent variables $(u,p)$. Consequently, the $(u,p)$ phase planes of (\ref{ODEsystem_desingularised_0}) and (\ref{ODEsystem_desingularised_1}) will have the same trajectories but the ``time'' it takes to evolve along such a trajectory is different. In particular,
when $D(u)>0$, $d\xi/dz>0$ and therefore trajectories on the phase planes of (\ref{ODEsystem_desingularised_0}) and (\ref{ODEsystem_desingularised_1}) have the same orientation. In contrast, when $D(u)<0$, $d\xi/dz<0$ and trajectories on the two phase planes are in the opposite direction, see Figure \ref{phaseplane_constrast1}. 
\begin{figure}
\centering
\begin{tikzpicture}[x=5cm,y=5cm, scale=1]
\fill[green!20] (0.5,0.2)--(0.83,0.2)--(0.83,-1)--(0.5,-1);
\draw[thick] (0,-1) -- (0,0.2)
node[above,thick] {$p$} coordinate( y axis);
\draw[very thick,blue] (0,0) -- (1,0) 
node[black,right] {$u$} coordinate( x axis);
\draw [color=blue,very thick] [label] plot file {phaseplane9_2.dat};
\draw[dashed,thick,color=red] (0.83,0.2) -- (0.83,-1);
\draw[dashed,thick,color=red] (0.5,0.2) -- (0.5,-1);
\draw[-latex,color=red] (0.4,-0.1) -- (0.2,-0.4);
\draw[-latex,color=red] (0.2,0.2) -- (0.3,0.1);
\draw[-latex,color=red] (0.45,-0.9) -- (0.4,-0.6);
\draw[-latex,color=red] (0.7,0.1)--(0.6,0.01) ;
\draw[-latex,color=red] (0.8,0.1)--(0.7,0.2) ;
\draw[-latex,color=red] (0.85,0.2) -- (0.95,0.1);
\draw[-latex,color=red] (0.95,0) -- (0.89,-0.05);
\draw[-latex,color=red] (0.55,-0.2)--(0.65,-0.7) ;
\draw[-latex,color=red] (0.95,-0.7) -- (0.85,-0.2);
\draw[shift={(0,0)}] (0.01,0) -- (-0.01,0) node[left] {$0$};
\draw[shift={(0,-1)}] (0.01,0) -- (-0.01,0) node[left] {$-0.02$};
\node [black,right] at (0.5,-0.04) {$\alpha$};
\node [black,left] at (0.83,-0.04) {$\beta$};
\node [black,below] at (0.5,-1) {(a)};
\end{tikzpicture}
\hfill
\begin{tikzpicture}[x=0.5cm,y=5cm, scale=1]
\draw[transparent] (0,-1) -- (0,0.2);
\draw[-latex,color=black] (0,-0.2)--(1,-0.2);
\node[black,left] at (0,-0.2) {$z$};
\node[black,right] at (1,-0.2) {$\xi$};
\end{tikzpicture}
\hfill
\begin{tikzpicture}[x=5cm,y=5cm, scale=1]
\fill[green!20] (0.5,0.2)--(0.83,0.2)--(0.83,-1)--(0.5,-1);
\draw[thick] (0,-1) -- (0,0.2)
node[above,very thick] {$p$} coordinate( y axis);
\draw[very thick,blue] (0,0) -- (1,0) 
node[black,right] {$u$} coordinate( x axis);
\draw [color=blue,very thick] [label] plot file {phaseplane9_2.dat};
\draw[dashed,thick,color=red] (0.83,0.2) -- (0.83,-1);
\draw[dashed,thick,color=red] (0.5,0.2) -- (0.5,-1);
\draw[-latex,color=red] (0.4,-0.1) -- (0.2,-0.4);
\draw[-latex,color=red] (0.2,0.2) -- (0.3,0.1);
\draw[-latex,color=red] (0.45,-0.9) -- (0.4,-0.6);
\draw[-latex,color=red] (0.6,0.01) -- (0.7,0.1);
\draw[-latex,color=red] (0.7,0.2) -- (0.8,0.1);
\draw[-latex,color=red] (0.85,0.2) -- (0.95,0.1);
\draw[-latex,color=red] (0.95,0) -- (0.89,-0.05);
\draw[-latex,color=red] (0.65,-0.7) -- (0.55,-0.2);
\draw[-latex,color=red] (0.95,-0.7) -- (0.85,-0.2);
\draw[shift={(0,0)}] (0.01,0) -- (-0.01,0) node[left] {$0$};
\draw[shift={(0,-1)}] (0.01,0) -- (-0.01,0) node[left] {$-0.02$};
\node [black,right] at (0.5,-0.04) {$\alpha$};
\node [black,left] at (0.83,-0.04) {$\beta$};
\node [black,below] at (0.5,-1) {(b)};
\end{tikzpicture}
\caption{(a) is the phase plane of system (\ref{ODEsystem_desingularised_0}) with parameters $D_i=0.25$, $D_g=0.05$, $\lambda=0.75$ and $c=0.866$. The vertical dashed lines are the walls of singularities $u=\alpha$ and $u=\beta$ and the solid blue lines are nullclines. Red arrows show the orientation of the trajectories. (b) is the phase plane of system (\ref{ODEsystem_desingularised_1}) for the same parameter values and red lines are nullclines. For $u$ in between $\alpha$ and $\beta$, the orientation of the trajectories is opposite compared to (a), while the orientation is the same for $u<\alpha$ and $u>\beta$.}
\label{phaseplane_constrast1}
\end{figure}
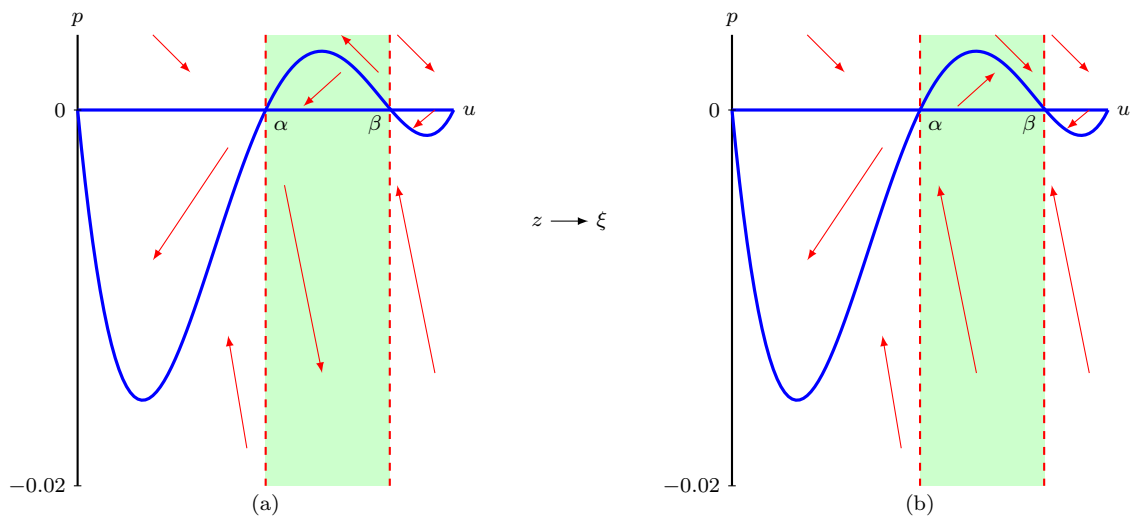 
Therefore, heteroclinic orbits of (\ref{ODEsystem_desingularised_0}) connecting $(1,0)$ to $(0,0)$ crossing the holes in the walls $(\alpha,0)$ and $(\beta,0)$, if they exist, are transformed and separated as heteroclinic orbits connecting $(1,0)$ to $(\beta,0)$, $(\alpha,0)$ to $(\beta,0)$ and $(\alpha,0)$ to $(0,0)$ of (\ref{ODEsystem_desingularised_1}) and \emph{vice versa}. Next, we will prove the existence of these heteroclinic orbits in system (\ref{ODEsystem_desingularised_1}) for a range of wave speeds $c$, and then combine these heteroclinic orbits in system (\ref{ODEsystem_desingularised_1}) as one global heteroclinic orbit in system (\ref{ODEsystem_desingularised_0}).

\subsection{Phase plane analysis of the desingularised system}
We first study the desingularised system (\ref{ODEsystem_desingularised_1}). It has nullclines $p=0$ and 
\begin{equation}
\label{nullcline1}
    p=-\frac{D(u)R(u)}{c}.
\end{equation}
The intersections of the two nullclines give four equilibrium points: $(0,0),(1,0),(\alpha,0),(\beta,0)$.
\begin{lemma} \label{L:EQ}
The equilibrium points $(1,0)$ and $(\alpha,0)$ are saddles. The equilibrium point $(0,0)$ is a stable node if
\begin{equation}
\label{condition2}
    c\ge2\sqrt{D(0)R'(0)}=2\sqrt{\lambda D_i}=c^*,
\end{equation}
and a stable spiral otherwise. The equilibrium point $(\beta,0)$ is a stable node if 
\begin{equation}
\label{condition1}
    c\ge2\sqrt{D'(\beta)R(\beta)},
\end{equation}
and a stable spiral otherwise.
\end{lemma}
\begin{proof}
The Jacobian of system (\ref{ODEsystem_desingularised_1}) is
\begin{equation}
\label{XJ}
   J(u,p)=\left(\begin{matrix}
   0 & 1 \\
   -F(u) & -c
  \end{matrix}\right),\quad\text{where}\quad F(u):=\frac{d}{du}\left(D(u)R(u)\right)=D'(u)R(u)+D(u)R'(u),
\end{equation}
with $D(u)R(u)$ the pointwise product of $D(u)$ and $R(u)$ and where we, as usual, omit the dot. The Jacobian has eigenvalues and eigenvectors
\begin{equation}
\nonumber
    \lambda_{\pm}=\frac{-c\pm\sqrt{c^2-4F(u)}}{2},\quad E_{\pm}=(1,\lambda{\pm}).
\end{equation}
For the equilibrium point $(1,0)$ this reduces to
\begin{equation}
    \label{eigenvector1}
    \lambda_{1\pm}=\frac{-c\pm\sqrt{c^2-4D(1)R'(1)}}{2},\quad E_{1\pm}=(1,\lambda_{1\pm}).
\end{equation}
The eigenvalues $\lambda_{1\pm}$ are real and of opposite sign since $D(1)=D_g>0$ and $R'(1)=-\lambda<0$. Thus $(1,0)$ is a saddle.

Similarly, the Jacobian of the equilibrium point $(\alpha,0)$ has eigenvalues and eigenvectors
\begin{equation}
    \label{eigenvectoralpha}
    \lambda_{\alpha\pm}=\frac{-c\pm\sqrt{c^2-4D'(\alpha)R(\alpha)}}{2},\quad E_{\alpha\pm}=(1,\lambda_{\alpha\pm}).
\end{equation}
Knowing that $D'(\alpha)<0$ and $R(\alpha)>0$, $\lambda_{\alpha+}$ is real and positive and $\lambda_{\alpha-}$ is real and negative. Thus $(\alpha,0)$ is a saddle.

The Jacobian of the equilibrium point $(0,0)$ has eigenvalues and eigenvectors
\begin{equation}
    \label{eigenvector0}
    \lambda_{0\pm}=\frac{-c\pm\sqrt{c^2-4D(0)R'(0)}}{2},\quad E_{0\pm}=(1,\lambda_{0\pm}).
\end{equation}
The eigenvalues $\lambda_{0\pm}$ are real and negative if (\ref{condition2}) holds since $D(0)=D_i>0$ and $R'(0)=\lambda>0$. Thus the equilibrium point $(0,0)$ is a stable node if (\ref{condition2}) holds. Otherwise, $\lambda_{0\pm}$ are complex-valued with negative real parts and $(1,0)$ is a stable spiral.

Similarly, the Jacobian of equilibrium point $(\beta,0)$ has eigenvalues and eigenvectors
\begin{equation}
    \label{eigenvectorbeta}
    \lambda_{\beta\pm}=\frac{-c\pm\sqrt{c^2-4D'(\beta)R(\beta)}}{2},\quad E_{\beta\pm}=(1,\lambda_{\beta\pm}).
\end{equation}
The eigenvalues $\lambda_{\beta\pm}$ are real and negative if (\ref{condition1}) holds since $D'(\beta)>0$ and $R(\beta)>0$. Thus the equilibrium point $(\beta,0)$ is a stable node if (\ref{condition1}) holds. Otherwise, $\lambda_{\beta\pm}$ are complex-valued with negative real parts and $(\beta,0)$ is a stable spiral.$\hfill\square$
\end{proof}
\begin{lemma} \label{L:THRES}
For $D_i>4D_g$, the thresholds of conditions (\ref{condition2}) and (\ref{condition1}) are ordered as
\begin{equation}
    c^* > 
    2\sqrt{D'(\beta)R(\beta)}.
    \label{relation1}
\end{equation}
\end{lemma}
\begin{proof}
The right hand side of (\ref{relation1}) is given by
\begin{equation}
\nonumber
    2\sqrt{D'(\beta)R(\beta)}=2\sqrt{3\lambda(D_i-D_g)\beta(1-\beta)(\beta-\alpha)}.
\end{equation}
Since $c^*=2\sqrt{\lambda D_i}$, proving relation (\ref{relation1}) is equivalent to proving
\begin{equation}
\nonumber
    D_i > 
    3(D_i-D_g)\beta(1-\beta)(\beta-\alpha),
\end{equation}
which is equivalent to proving
\begin{equation}
\label{relation21}
    \frac{D_i}{D_i-D_g}> 
    3\beta(1-\beta)(\beta-\alpha).
\end{equation}
Knowing that $2/3<\beta<1$ and $0<\beta-\alpha<2/3$ gives $3\beta(1-\beta)(\beta-\alpha)<2/3$.
Since $D_i>4D_g$, we have that ${D_i}/{(D_i-D_g)}>1$ since $D_i>D_i-D_g$. Hence, (\ref{relation21}) holds and thus (\ref{relation1}) holds.
$\hfill\square$
\end{proof}

For $c<c^*$, $(0,0)$ becomes a spiral node and hence we expect trajectories approaching $(0,0)$ to become negative which in the end would lead to travelling wave solutions become negative. Therefore, we now assume that $c\ge c^*$. To prove the existence of heteroclinic orbits between the equilibrium points, we construct invariant regions in the phase plane from which trajectories cannot leave, so that the Poincar\'{e}-Bendixson theorem can be applied \citep{jordan1999nonlinear}, see Figure \ref{phaseplane_1}. The slope of nullcline $(\ref{nullcline1})$ is $\chi(u)=-F(u)/c$, where $F(u)$ is given by (\ref{XJ}), while the slope of the unstable eigenvector of $(1,0)$ is $\lambda_{1+}$, see (\ref{eigenvector1}). We thus have
\begin{equation}
\label{relation22}
  \begin{aligned}
    \lambda_{1+}-\chi(1)=&\frac{-c+\sqrt{c^2-4D(1)R'(1)}}{2}+\frac{1}{c}D(1)R'(1)\\
    =&\frac{c\sqrt{c^2-4D(1)R'(1)}-\left(c^2-2D(1)R'(1)\right)}{2c}\\
    =&\frac{\sqrt{c^4-4c^2D(1)R'(1)}-\sqrt{c^4-4c^2D(1)R'(1)+4\left(D(1)R'(1)\right)^2}}{2c}<0.
  \end{aligned}
\end{equation}
That is, the unstable eigenvector of $(1,0)$ has a smaller slope than nullcline (\ref{nullcline1}) at $(1,0)$. In other words, the trajectory leaving $(1,0)$ with decreasing $u$ initially lies above the nullcline (\ref{nullcline1}).

Similarly, the slope of the unstable eigenvector of $(\alpha,0)$ is $\lambda_{\alpha+}$, see (\ref{eigenvectoralpha}). We have, after similar computation as (\ref{relation22}), $\lambda_{\alpha+}-\chi(\alpha)<0$.
Thus, the unstable eigenvector of $(\alpha,0)$ has a smaller slope than nullcline (\ref{nullcline1}) at $(\alpha,0)$. Therefore, the trajectory leaving $(\alpha,0)$ with decreasing $u$ initially lies above the nullcline (\ref{nullcline1}), while the trajectory leaving $(\alpha,0)$ with increasing $u$ initially lies below the nullcline (\ref{nullcline1}).

Under condition (\ref{condition2}), the least negative slope of the stable eigenvectors of equilibrium point $(0,0)$ is $\lambda_{0+}$, see (\ref{eigenvector0}). This gives, after a similar computation as $(\ref{relation22})$, $\lambda_{0+}-\chi(0)<0$.
Thus, both eigenvectors of $(0,0)$ have slopes that are more negative than nullcline (\ref{nullcline1}) at $(0,0)$. In other words, the eigenvectors of $(0,0)$ initially lie under the nullcline (\ref{nullcline1}) for $u>0$.

Similarly, under condition (\ref{condition1}), the least negative slope of the stable eigenvectors of $(\beta,0)$ is $\lambda_{\beta+}$, see (\ref{eigenvectorbeta}). This gives $\lambda_{\beta+}-\chi(\beta)<0$.
Thus, both eigenvectors have slopes that are more negative than nullcline (\ref{nullcline1}) at $(\beta,0)$. Therefore, the trajectory moving in $(\beta,0)$ with decreasing $u$ initially lies under the nullcline (\ref{nullcline1}) for $u>\beta$, while they lie above the nullcline (\ref{nullcline1}) for $u<\beta$, see also Figure \ref{phaseplane_1}.

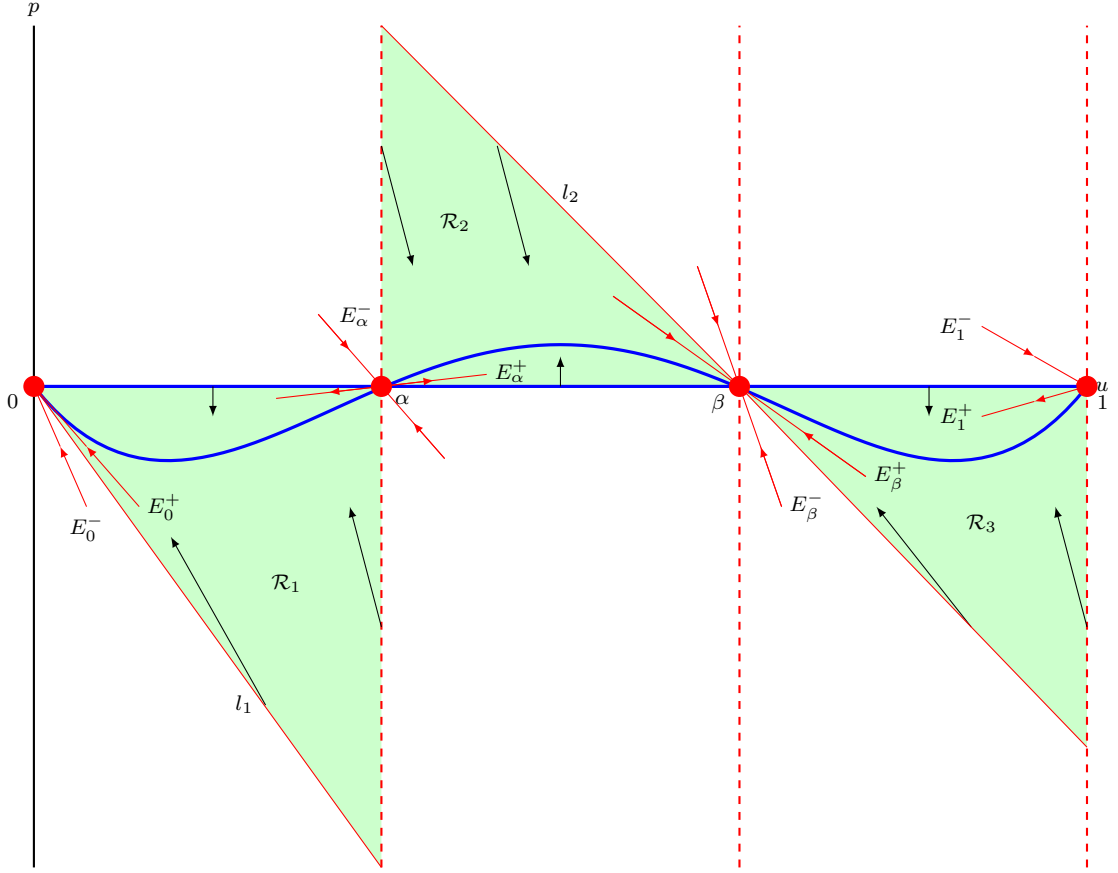
\begin{figure}
\centering
\begin{tikzpicture}[x=7cm,y=0.4cm, scale=2]
\fill[green!20] (0,0)--(0.33,0)--(0.33,-8);
\fill[green!20] (0.33,0)--(0.67,0)--(0.33,6);
\fill[green!20] (0.67,0)--(1,0)--(1,-6);
\draw [color=blue,very thick] [label] plot file {phaseplane8_1.dat};
\draw[thick] (0,-8) -- (0,6) node[above,very thick] {$p$} coordinate(y axis);
\draw[very thick,blue] (0,0) -- (1,0) 
node[black,right, thick] {$u$} coordinate( x axis);
\node [black,below] at (1.015,-0.01) {$1$};
\node [black,below] at (-0.02,0) {$0$};
\node [black,below] at (0.35,0) {$\alpha$};
\node [black,below] at (0.65,0) {$\beta$};
\node [black,below] at (0.24,-3) {$\mathcal{R}_1$};
\node [black,below] at (0.4,3) {$\mathcal{R}_2$};
\node [black,below] at (0.9,-2) {$\mathcal{R}_3$};
\node [black,below] at (0.2,-5) {$l_1$};
\node [black,below] at (0.51,3.5) {$l_2$};
\draw[dashed,color=red,thick] (0.33,-8) -- (0.33,6);
\draw[dashed,color=red,thick] (0.67,-8) -- (0.67,6);
\draw[dashed,color=red,thick] (1,-8) -- (1,6);
\draw[-latex] (0.22,-5.3) -- (0.13,-2.5);
\draw[-latex] (0.33,-4) -- (0.3,-2);
\draw[-latex] (0.17,0) -- (0.17,-0.5);
\draw[-latex] (0.5,0) -- (0.5,0.5);
\draw[-latex] (0.85,0) -- (0.85,-0.5);
\draw[-latex] (0.33,4) -- (0.36,2);
\draw[-latex] (0.44,4) -- (0.47,2);
\draw[-latex] (0.89,-4) -- (0.8,-2);
\draw[-latex] (1,-4) -- (0.97,-2);
\draw[color=red,-latex] (1,0) -- (0.95,-0.25);
\draw[color=red] (0.95,-0.25) -- (0.9,-0.5);
\draw[-latex,red] (0.9,1) -- (0.95,0.5);
\draw[-latex,red] (0.95,0.5) -- (1,0);

\draw[color=red] (0.43,0.2) -- (0.23,-0.2);
\draw[color=red,-latex] (0.33,0) -- (0.38,0.1);
\draw[color=red,-latex] (0.33,0) -- (0.28,-0.1);

\draw[color=red] (0.39,-1.2) -- (0.27,1.2);
\draw[color=red,-latex] (0.39,-1.2) -- (0.36,-0.6);
\draw[color=red,-latex] (0.27,1.2) -- (0.3,0.6);

\draw[color=red] (0,0) -- (0.33,-8);
\draw[color=red] (0.67,0) -- (1,-6);
\draw[color=red] (0.33,6) -- (0.67,0);
\draw[red] (0,0) -- (0.05,-1);
\draw[-latex,red] (0.1,-2) -- (0.05,-1);
\draw[red] (0,0) -- (0.025,-1);
\draw[-latex,red] (0.05,-2) -- (0.025,-1);
\draw[color=red] (0.55,1.5) -- (0.79,-1.5);
\draw[color=red,-latex] (0.55,1.5) -- (0.61,0.75);
\draw[color=red,-latex] (0.79,-1.5) -- (0.73,-0.75);
\draw[color=red] (0.63,2) -- (0.71,-2);
\draw[color=red,-latex] (0.63,2) -- (0.65,1);
\draw[color=red,-latex] (0.71,-2) -- (0.69,-1);
\node [black,right] at (0.1,-2) {$E_0^{+}$};
\node [black,below] at (0.05,-2) {$E_0^{-}$};
\node [black,right] at (0.43,0.26) {$E_\alpha^{+}$};
\node [black,left] at (0.33,1.2) {$E_\alpha^{-}$};
\node [black,right] at (0.79,-1.5) {$E_\beta^{+}$};
\node [black,right] at (0.71,-2) {$E_\beta^{-}$};
\node [black,left] at (0.9,1) {$E_1^{-}$};
\node [black,left] at (0.9,-0.5) {$E_1^{+}$};
\fill[red] (0.33,0) circle (2pt);
\fill[red] (0.67,0) circle (2pt);
\fill[red] (1,0) circle (2pt);
\fill[red] (0,0) circle (2pt);
\end{tikzpicture}
\caption{A qualitative phase plane of system (\ref{ODEsystem_desingularised_1}). The three dashed lines are $u=\alpha$, $u=\beta$ and $u=1$. The blue lines are the nullclines $p=0$ and $p=-D(u)R(u)/c$. Region $\mathcal{R}_1$ is bounded by $p=0$, $u=\alpha$ and a straight line $l_1$ with negative slope passing through $(0,0)$. Region $\mathcal{R}_2$ is bounded by $p=0$, $u=\alpha$ and a straight line $l_2$ with negative slope passing through $(\beta,0)$. Region $\mathcal{R}_3$ is bounded by $p=0$, $u=1$ and $l_2$.} 
\label{phaseplane_1}
\end{figure}

Next, we consider the region $\mathcal{R}_1$ bounded by $p=0$, $u=\alpha$ and a straight line $l_1$ through $(0,0)$ with a negative slope $\mu_1$. We aim to prove that for $c\ge c^*$, there always exists a slope $\mu_1$ so that no trajectories in region $\mathcal{R}_1$ can cross through its boundaries. Trajectories starting on $p=0$ have negative vertical directions since $du/d\xi=p=0$ and $dp/d\xi=-D(u)R(u)<0$ for $u\in(0,\alpha)$. Thus, trajectories in $\mathcal{R}_1$ cannot cross through $p=0$. Trajectories starting on $u=\alpha$ with negative $p$ values point into region $\mathcal{R}_1$ since $du/d\xi=p<0$ and $dp/d\xi=-cp>0$. Trajectories starting on $l_1$ satisfy $p=\mu_1 u$, and they point into $\mathcal{R}_1$ only if
\begin{equation}
    \nonumber
    \frac{dp}{du}\Bigr\rvert_{p=\mu_1u}=-c-\frac{D(u)R(u)}{\mu_1u}\le\mu_1,\quad \text{for}\quad u\in(0,\alpha).
\end{equation}
After rearranging and recalling that $\mu_1<0$, we obtain
\begin{equation}
\label{condition_wow1}
    \mu_1(\mu_1+c)\le-\frac{D(u)R(u)}{u} = -\lambda D(u)(1-u),\quad \text{for}\quad u\in(0,\alpha).
\end{equation}
\begin{lemma} \label{L:S1}
For $c\ge c^*$, there exists a $\mu_1$ such that inequality (\ref{condition_wow1}) is valid for any $u\in(0,\alpha)$.
\end{lemma}
\begin{proof} 
Proving inequality (\ref{condition_wow1}) is equivalent to proving 
\begin{equation}
\label{condition_wow1_1}
    \mu_1(\mu_1+c)\le- \lambda \sup_{u\in(0,\alpha)} D(u)(1-u). 
\end{equation}
The left hand side of inequality (\ref{condition_wow1_1}) is minimal when $\mu_1=-c/2$. Setting $\mu_1=-c/2$ and substituting into inequality (\ref{condition_wow1_1}) gives a lower bound
\begin{equation}
\label{condition_wow1_2}
    c_1=2 \sqrt{\lambda} \sup_{u\in(0,\alpha]}\sqrt{D(u)(1-u)},
\end{equation}
such that $(\ref{condition_wow1_1})$ holds for $c\ge c_1$. The right hand side of (\ref{condition_wow1_2}) gives
\begin{equation}
\nonumber
    2 \sqrt{\lambda} \sup_{u\in(0,\alpha)}\sqrt{D(u)(1-u)}=2\sqrt{\lambda D(0)}=2\sqrt{\lambda D_i},
\end{equation}
since $D(u)$ and $(1-u)$ are both decreasing functions on $u\in(0,\alpha)$. Thus, $c_1=c^*$. Hence, for $c\ge c^*$, inequality (\ref{condition_wow1_1}) is valid for $\mu_1=-c/2$.$\hfill\square$
\end{proof}

Knowing that for $c\ge c^*$ inequality (\ref{condition_wow1}) is valid, trajectories on $l_1$ with $\mu_1=-c/2$ point into region $\mathcal{R}_1$. Thus, based on the Poincar\'{e}-Bendixson theorem \citep{jordan1999nonlinear}, the observation that the derivative of $u$ is negative in the region $\mathcal{R}_1$ (preventing the existence of a homoclinic orbit) and the absence of fixed points in the interior of $\mathcal{R}_1$ (preventing the existence of a limit cycle), the trajectory leaving from the equilibrium point $(\alpha,0)$ with decreasing $u$ and decreasing $p$ must connect with the equilibrium point $(0,0)$ without going negative in $u$. 

Similarly, we consider the region $\mathcal{R}_2$ bounded by $p=0$, $u=\alpha$ and a straight line $l_2$ through $(\beta,0)$ with a negative slope $\mu_2$, and the region $\mathcal{R}_3$ bounded by $p=0$, $u=1$ and $l_2$. Trajectories starting on $p=0$ have positive vertical directions for $u\in(\alpha, \beta)$ since $du/d\xi=p=0$ and $dp/d\xi=-D(u)R(u)>0$  and they have negative vertical directions since for $u\in(\beta,1)$, $du/d\xi=0$ and $dp/d\xi=-D(u)R(u)<0$. Trajectories starting on $u=\alpha$ with positive $p$ point into region $\mathcal{R}_2$ since $du/d\xi=p>0$ and $dp/d\xi=-cp<0$. Similarly, trajectories starting on $u=1$ with negative $p$ point into region $\mathcal{R}_3$. In addition, requiring the existence of a slope $\mu_2$ such that trajectories starting on $l_2$ point into regions $\mathcal{R}_2$ and $\mathcal{R}_3$ leads to the condition
\begin{equation}
\label{condition_wow2}
    \mu_2(\mu_2+c)\le-\frac{D(u)R(u)}{u-\beta}=-3(D_i-D_g)(u-\alpha)R(u), \quad \text{for} \quad u\in(\alpha,1). 
\end{equation}
\begin{lemma} \label{L:S2}
For $c\ge c^*$, there exists a $\mu_2$ such that inequality (\ref{condition_wow2}) is valid for any $u\in(\alpha,1)$.
\end{lemma}
\begin{proof}
The proof of Lemma~\ref{L:S2} is analogous to the proof of Lemma~\ref{L:S1} and we will omit some of the details. Again, there exist a lower bound  
\begin{equation} \nonumber
    c_2 
    =2 \sqrt{3(D_i-D_g)} \sup_{u\in(\alpha,1)}\sqrt{(u-\alpha)R(u)},
\end{equation}
such that (\ref{condition_wow2}) holds for $c\ge c_2$. Next, we show that $c_2<c^*$. That is, we show that
\begin{equation}
\nonumber
    2\sqrt{\lambda D_i}> 2 \sqrt{3(D_i-D_g)} \sup_{u\in(\alpha,1)}\sqrt{(u-\alpha)R(u)}.
\end{equation}
This is equivalent to proving $D_i/(D_i-D_g)>3u(1-u)(u-\alpha)$
for $u\in(\alpha,1)$. Noticing that $u-\alpha<2/3$, and $u(1-u)\le1/4$, we obtain $3u(1-u)(u-\alpha)<1/2$. Subsequently, we have
\begin{equation}
\nonumber
    \frac{D_i}{D_i-D_g}>1>\frac{1}{2}>3u(1-u)(u-\alpha),
\end{equation}
since $D_i>4D_g$ by assumption. Thus, $c_2<c^*$.$\hfill\square$
\end{proof}

Knowing that for $c\ge c^*$ the inequality (\ref{condition_wow2}) is valid, trajectories on $l_2$ in between $\alpha$ and $\beta$ point into region $\mathcal{R}_2$. Thus, based on the Poincar\'{e}-Bendixson theorem \citep{jordan1999nonlinear}, the trajectory leaving from the equilibrium point $(\alpha,0)$ with increasing $u$ and increasing $p$ must connect with the equilibrium point $(\beta,0)$. Analogously, the trajectory leaving from the equilibrium point $(1,0)$ with decreasing $u$ and decreasing $p$ must connect with the equilibrium point $(\beta,0)$.

In summary, for $c\ge c^*$ there exist heteroclinic orbits connecting $(1,0)$ to $(\beta,0)$, $(\alpha,0)$ to $(\beta,0)$ and $(\alpha,0)$ to $(0,0)$ in system (\ref{ODEsystem_desingularised_1}). Since trajectories in $u\in(0,\alpha)\cup(\beta,0)$ in system (\ref{ODEsystem_desingularised_0}) are the same, and have the same orientation, as in system (\ref{ODEsystem_desingularised_1}), there exist trajectories connecting $(1,0)$ to the hole in the wall $(\beta,0)$ and trajectories connecting the hole in the wall $(\alpha,0)$ to $(0,0)$ in system (\ref{ODEsystem_desingularised_0}). For $u\in(\alpha,\beta)$, trajectories of system (\ref{ODEsystem_desingularised_0}) move in opposite direction compared to $(\ref{ODEsystem_desingularised_1})$, see Figure~\ref{phaseplane_constrast1}. The trajectory leaving from $(\alpha,0)$ with increasing $u$, positive $p$ and connecting to $(\beta,0)$ in system (\ref{ODEsystem_desingularised_1}) becomes a trajectory leaving from $(\beta,0)$ with decreasing $u$, positive $p$ and connecting to $(\alpha,0)$ in system (\ref{ODEsystem_desingularised_0}). Thus, there exists an orbit connecting $(\beta,0)$ to $(\alpha,0)$ in system (\ref{ODEsystem_desingularised_0}). Combining the above, we get that for $c \ge c^*$, there exists a heteroclinic orbit with $u\ge0$ connecting $(1,0)$ to $(0,0)$ passing through holes in the walls $(\alpha,0)$ and $(\beta,0)$ in system (\ref{ODEsystem_desingularised_0}), however, see Remark~\ref{REM}. Hence, there exist smooth monotone travelling wave solutions of (\ref{RDE_1}) with positive speed $c\ge c^*$. This completes the proof of Theorem~\ref{THEOREM}.

For $2\sqrt{D'(\beta)R(\beta)} < c<c^*$ the equilibrium point $(\beta,0)$ of the desingularised system (\ref{ODEsystem_desingularised_1}) is still a stable node, while $(0,0)$ is a stable spiral, see Lemma~\ref{L:EQ}. We can use similar techniques as above to show that system (\ref{ODEsystem_desingularised_1}) still possesses heteroclinic orbits connecting $(1,0)$ to $(\beta,0)$, $(\alpha,0)$ to $(\beta,0)$ and $(\alpha,0)$ to $(0,0)$, see also Figure \ref{phaseplane_pic3}. However, this latter heteroclinic orbit now spirals into $(0,0)$. Consequently, also for $2\sqrt{D'(\beta)R(\beta)} < c<c^*$ there exists a heteroclinic orbit connecting $(1,0)$ to $(0,0)$ passing through holes in the walls $(\alpha,0)$ and $(\beta,0)$ in system (\ref{ODEsystem_desingularised_0}). However, these correspond to smooth travelling wave solutions of (\ref{RDE_1}) with (\ref{D(u)2}) and (\ref{R(u)2}) that are not monotone and instead oscillate around $0$. These
solutions are not biologically relevant as $U$ represents the population density in the discrete model and thus cannot be negative. 

For $0<c<2\sqrt{D'(\beta)R(\beta)}$, $(\beta,0)$ becomes a stable spiral in (\ref{ODEsystem_desingularised_1}) and hence trajectories in system (\ref{ODEsystem_desingularised_0}) can no longer pass through this hole in the wall, i.e. the hole in the wall is not of the correct type \citep{harley2014existence}. That is, (\ref{RDE_1}) with (\ref{D(u)2}) and (\ref{R(u)2}) do not support
smooth 
travelling wave solutions for $0<c<2\sqrt{D'(\beta)R(\beta)}$. Note that there may exist shock-fronted travelling wave solutions, however, we are not interested in such solutions in this manuscript as $(0,0)$ is still a stable spiral of  (\ref{ODEsystem_desingularised_1}) and thus again yields solutions that are not biologically relevant. See Section~\ref{SS:shocks} for a further discussion related to shock-fronted travelling wave solutions supported by (\ref{RDE_1}).
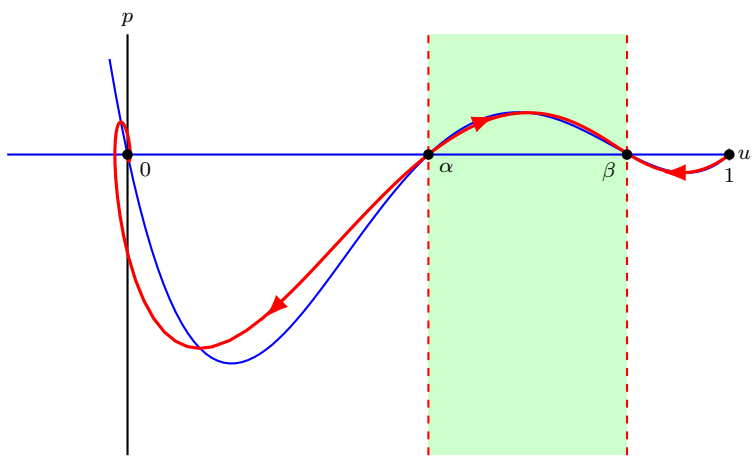
\begin{figure}
\centering
\begin{tikzpicture}[x=8cm,y=80cm, scale=1]
\fill[green!20] (0.5,0.02)--(0.83,0.02)--(0.83,-0.05)--(0.5,-0.05);
\draw[dashed,red,thick] (0.5,-0.05) -- (0.5,0.02); 
\draw[dashed,red,thick] (0.83,-0.05) -- (0.83,0.02); 
\draw[thick] (0,-0.05) -- (0,0.02) node[above,thick] {$p$} coordinate(y axis);
\draw[blue,thick] (-0.2,0) -- (1,0) 
node[black,right,thick] {$u$} coordinate( x axis);
\foreach \pos in {1}  
\draw[shift={(\pos,0)}] (0,0.001) -- (0,-0.001) node[below] {$\pos$};
\draw [color=blue,thick] [label] plot file {Heteroclinicorbit4.dat};
\draw [color=red,very thick] [label] plot file {Heteroclinicorbit1.dat};
\draw [color=red,very thick] [label] plot file {Heteroclinicorbit2.dat};
\draw [color=red,very thick] [label] plot file {Heteroclinicorbit3.dat};
\node[black,below] at (0.03,0) {0};
\node[black,below] at (0.53,0) {$\alpha$};
\node[black,below] at (0.8,0) {$\beta$};
\fill[black] (0,0) circle (2pt);
\fill[black] (0.5,0) circle (2pt);
\fill[black] (0.83,0) circle (2pt);
\fill[black] (1,0) circle (2pt);
\draw[color=red,-{Latex[length=3mm]}] (0.6,0.006) -- (0.61,0.0064);
\draw[color=red,-{Latex[length=3mm]}] (0.9,-0.003) -- (0.89,-0.003);
\draw[color=red,-{Latex[length=3mm]}] (0.24,-0.026) -- (0.23,-0.027);
\end{tikzpicture}
\caption{Phase plane of system (\ref{ODEsystem_desingularised_1}) with parameters $D_i=0.25$, $D_g=0.05$, $\lambda=0.75$ and $c=0.4$. The latter is smaller than $c^*\approx0.866$ but larger than $2\sqrt{D'(\beta)R(\beta)} \approx 0.289$. The blue lines are the nullclines $p=0$ and $p=-D(u)R(u)/c$. The red lines are the heteroclinic orbits connecting $(0,0)$, $(\alpha,0)$, $(\beta,0)$, and $(1,0)$.}
\label{phaseplane_pic3}
\end{figure}

\begin{remark} \label{REM}
It is important to note that combining the three heteroclinic orbits in the desingularised system (\ref{ODEsystem_desingularised_1}) to get the global one in the original system (\ref{ODEsystem_desingularised_0}) is not trivial. Although the relationship between the trajectories, and their orientation, in the two systems is clear, we still need to prove that orbits are able to pass through the holes in the wall in (\ref{ODEsystem_desingularised_0}) by, for instance, using the canard theory \citep{szmolyan2001canards,wechselberger2005existence,wechselberger2012propos}. Roughly speaking, we embed the original ODE (\ref{ODE_1}) into a larger class of problems by adding a higer order perturbation term with a small parameter $0\leq\epsilon\ll1$. Subsequently, rather than obtaining the two-dimensional system (\ref{ODEsystem_desingularised_0}), we have a higher-dimensional system which has a \emph{slow-fast} structure that can be studied by geometric singular perturbation theory \citep{jones1995geometric}. Most notably, the two-dimensional system (\ref{ODEsystem_desingularised_0}) would become the reduced problem of the higher-dimensional system in the singular limit $\epsilon\to0$ and it is constraint on a folded critical manifold. With canard theory we can show the existence of solutions crossing through the holes in the wall (or folded canard points) in the higher-dimensional system for $0\le\epsilon\ll1$.
As this is by now relatively standard and straightforward, we decide to omit the details and instead refer to \citet{szmolyan2001canards,wechselberger2005existence,wechselberger2012propos}, and references therein.
\end{remark}

\section{Stability analysis} \label{S:ST}
We showed that, similar to the Fisher-KPP equation \cite[e.g.]{harley2015numerical}, (\ref{RDE_1}) with (\ref{D(u)2}) and (\ref{R(u)2}) supports smooth travelling wave solutions for $c > 2\sqrt{D'(\beta)R(\beta)} $, but that only the travelling wave solutions with $c \geq c^*$ (\ref{minimumwave}) have nonnegative densities.
The minimal wave speed for the Fisher-KPP equation is closely related to the onset of absolute instabilities\footnote{Note that there are several other ways, for instance with sub-solutions \citep{larson1978transient}, to show that the minimal wave speed for the Fisher-KPP equation is $c^*$.}.
Roughly speaking, absolute instabilities imply that perturbations to a travelling wave solution (in an appropriate Sobolev space that will be discussed further on) will grow for all time and at every point in space \citep{sherratt2014mathematical}. 
These instabilities are related to the absolute spectrum of the linear operator associated to the travelling wave solution and is fully determined by the asymptotic behaviour ($z \to \pm \infty$) of the travelling wave solution \citep{kapitula2013spectral,sandstede2002stability}.
Note that the absolute spectrum is, strictly speaking, not part of the spectrum of the linear operator. However, it gives an indication on how far the essential spectrum can be shifted to the left upon using a weighted Sobolev space \citep{kapitula2013spectral,sandstede2002stability}. Consequently, if parts of the absolute spectrum lie in the right half plane, then the essential spectrum cannot be fully weighted into the open left half plane, and the associate solution is hence absolutely unstable\footnote{See the introduction of \cite{davis2017absolute} for definitions, and an explicit computation, of the absolute spectrum for the Fisher-KPP equation.}.
The travelling wave solutions of (\ref{RDE_1}) with (\ref{D(u)2}) and (\ref{R(u)2}) as constructed in Section~\ref{S:EXIST} asymptote to $0$ and $1$ and the nonlinear diffusivity function $D(U)$ is positive near $U=0$ and $U=1$, see (\ref{signcondition1}). That is, near these points (\ref{RDE_1}) with (\ref{D(u)2}) and (\ref{R(u)2}) has a Fisher-KPP imprint and we therefore expect that the minimal wave speed $c^*$ of (\ref{RDE_1}) is also closely related to the onset of absolute instabilities. In other words, we expect that the travelling wave solutions of (\ref{RDE_1}) with (\ref{D(u)2}) and (\ref{R(u)2}) are absolutely unstable for $2\sqrt{D'(\beta)R(\beta)} < c<c^*$. Therefore, we expect perturbations to these travelling wave solutions to always grow and we will never observe them in, for instance, numerical simulations. Consequently, while (\ref{RDE_1}) with (\ref{D(u)2}) and (\ref{R(u)2}) support these biological irrelevant travelling wave solutions that go negative, they will never be observed and thus do not effect the feasibility of the model.

Below, we briefly describe how to determine the absolute spectrum of a travelling wave solution. For a more detailed and complete mathematical description, we refer to \cite{davis2017absolute,kapitula2013spectral} and \cite{sandstede2002stability}.
To determine the absolute spectrum of a travelling wave solution $\hat{u}(z)$, we add a small perturbation $q(z,t)$ to the travelling wave solution and determine how this perturbation evolves under the PDE in its moving frame. 
That is, we substitute $u(z,t)=\hat{u}(z)+q(z,t)$ into (\ref{stabilitysec2}) and, upon ignoring higher-order perturbative terms $\mathcal{O}(q^2)$, we get 
\begin{equation} \label{L} 
\begin{aligned}
\frac{\partial q}{\partial t}=\mathcal L q\,, \qquad \mathcal{L} := D(\hat{u})\frac{\partial ^2}{\partial z^2}+\left(2D'(\hat{u})\frac{d \hat{u}}{d z}+c\right)\frac{\partial }{\partial z}+\left(D'(\hat{u})\frac{d^2\hat{u}}{d z^2}+D''(\hat{u})\left(\frac{d \hat{u}}{d z}\right)^2+R'(\hat{u})\right).
\end{aligned}
\end{equation}
The associated eigenvalue problem, which is obtained by setting $q(z,t)=e^{\Lambda t}q(z)$, is given by
\begin{equation} \label{EIG}
    \mathcal{L} q=\Lambda q.
    \end{equation} 
The spectral stability of the travelling wave solution $\hat{u}$ is now determined by the spectrum of the linear operator $\mathcal{L}$, that is, the $\Lambda \in \mathbb{C}$ for which $\mathcal{L} - \Lambda$ is not invertible. By translation invariance $0$ is always an eigenvalue (with eigenfunction $\hat{u}'$) and we call the travelling wave solution $\hat{u}$ spectrally stable if the nonzero spectrum is in the open left half plane and unstable otherwise. This spectrum naturally breaks up into two sets, the
point spectrum and the essential spectrum \citep{kapitula2013spectral,sandstede2002stability}. Roughly speaking, the essential spectrum of the travelling wave solution deals with instabilities at infinity and it is related to the spectrum of the background linear operator $\mathcal{L}$ as $z \to \pm \infty$, while the point spectrum deals with the stability of the actual wave front. 
 
Obviously, the spectral properties of $\mathcal L$ depend on the space we allow the perturbations $q$ to be taken from. A natural choice is the space of square integrable functions whose first (weak) derivative (in $z$) is also square integrable, that is, the Sobolev space $\mathbb{H}^1(\mathbb{R})$. Another choice is the related one-sided weighted space $\mathbb{H}^1_{\nu}(\mathbb{R})$ defined as $q \in \mathbb{H}^1_{\nu}(\mathbb{R})$ if and only if $e^{\nu z}q \in \mathbb{H}^1(\mathbb{R})$ \citep{kapitula2013spectral,sattinger1977weighted}. For positive $\nu$ the weight forces $q$ to decay at a rate faster than $e^{-\nu z}$ as $z\to\infty$, while it is allowed to grow exponentially, but at a rate less than $e^{-\nu z}$, as $z\to-\infty$. That is, the weight provides information whether the travelling wave solution is more sensitive to perturbations at plus or minus infinity \citep{davis2017absolute}. 
The weighting of $\mathbb{H}^1(\mathbb{R})$ shifts the essential spectrum \citep{kapitula2013spectral}. That is, a travelling wave solution can be unstable with respect to perturbations in $\mathbb{H}^1(\mathbb{R})$, while it is stable with respect to perturbations in an appropriately weighted space $\mathbb{H}_{\nu}^1(\mathbb{R})$. This is, for instance, the case for the Fisher-KPP equation and a particular Keller-Segel model \citep{davis2017absolute,davis2018spectral}. The absolute spectrum of a travelling wave solution is not affected by the weighting of the space and gives an indication on how far the essential spectrum can be weighted (as the absolute spectrum is always to the left of the rightmost boundary of the essential spectrum \citep{davis2017absolute}). In other words, if the absolute spectrum of a travelling wave solution contains part of the right half plane then the essential spectrum cannot be weighted into the open left half plane and the travelling wave solution is said to be absolutely unstable.

The eigenvalue problem (\ref{EIG}) can be written as a system of first order ODEs 
\begin{equation} \nonumber
    \mathcal{T}(\Lambda)\left(\begin{aligned}
    &q\\
    &s
    \end{aligned}\right) :=\left(\frac{d}{dz}-A(z;\Lambda)\right)\left(\begin{aligned}
    &q\\
    &s
    \end{aligned}\right) = 0\,,
\quad
{\rm where} \quad
    A(z;\Lambda):=\begin{pmatrix}
    0&1\\
    \mathcal{B}&\mathcal{C}
      \end{pmatrix},
\end{equation}
with
\begin{equation}
\nonumber
    \begin{aligned}
    &\mathcal{B}=-\frac{1}{D(\hat{u})}\left(D'(\hat{u})\frac{d^2\hat{u}}{d z^2}+D''(\hat{u})\left(\frac{d \hat{u}}{d z}\right)^2+R'(\hat{u})-\Lambda\right),
    &\mathcal{C}=-\frac{1}{D(\hat{u})}\left(2D'(\hat{u})\frac{d \hat{u}}{d z}+c\right).
    \end{aligned}
\end{equation}
The unweighted essential spectrum and the absolute spectrum of the operator $\mathcal{L}$ are determined by the asymptotic behaviour of the operator $\mathcal{T}(\Lambda)$ since the operator is a relatively compact perturbation of the operator when you plug in $z = \pm \infty$ \citep{kapitula2013spectral}. Therefore, we define the asymptotic matrices
\begin{equation}
\nonumber
    A_+(\Lambda):=\lim_{z\to+\infty}A(z,\Lambda)=
   \begin{pmatrix}
    0&1\\
    \dfrac{-R'(0)+\Lambda}{D(0)}&-\dfrac{c}{D(0)}
   \end{pmatrix},
\end{equation}
and
\begin{equation}
\nonumber
    A_-(\Lambda):=\lim_{z\to-\infty}A(z,\Lambda)=
     \begin{pmatrix}
    0&1\\
    \dfrac{-R'(1)+\Lambda}{D(1)}&-\dfrac{c}{D(1)}
   \end{pmatrix}.
\end{equation}
More specifically, for the problem at hand the boundary of the unweighted essential spectrum of $\mathcal{L}$ is determined by those $\Lambda$ for which $A_{\pm}(\Lambda)$ has a purely imaginary eigenvalue.

In contrast, the
absolute spectrum at $\pm \infty$ is determined by those $\Lambda$ for which the eigenvalues of $A_{\pm}(\Lambda)$ have the same real part \citep{sandstede2002stability}.
The eigenvalues of $A_+$ are
\begin{equation}
\label{mu+}
    \mu_{+}^{\pm}=\frac{-c\pm\sqrt{c^2-4D(0)R'(0)+4D(0)\Lambda}}{2D(0)},
\end{equation}
and those of $A_{-}$ are
\begin{equation}
\label{mu-}
   \mu_{-}^{\pm}=\frac{-c\pm\sqrt{c^2-4D(1)R'(1)+4D(1)\Lambda}}{2D(1)}.
\end{equation}
Hence, the boundary of the unweighted essential spectrum is given by the so-called dispersion relations
\begin{equation} \nonumber
    \Lambda_+=-D(0)k^2+ick+R'(0), \quad {\rm and} \quad  \Lambda_-=-D(1)k^2+ick+R'(1),
\end{equation}
where $k\in\mathbb{R}$ and where $\mu_\pm^+=ik$ are the purely imaginary spatial eigenvalue of $A_\pm$.
These dispersion relations form two parabolas, opening leftward and intersecting the real axis at $R'(0)=\lambda>0$ and $R'(1)=-\lambda<0$, see Figure \ref{essentialspectrum1}. That is, all travelling wave solutions of (\ref{RDE_1}) with (\ref{D(u)2}) and (\ref{R(u)2}) have unweighted essential spectrum in the right half plane.
\begin{figure}
\centering
\begin{tikzpicture}[x=0.9cm,y=1cm, scale=1]
\fill[green!20] (1,0)--(0.91,-0.3)--(0.75,-0.5)--(0.51,-0.7)--(0.36,-0.8)--(-0.21,-1.1)--(-0.69,-1.3)--(-1.89,-1.7)--(-3,-2)--(-5,-2)--(-3.89,-1.7)--(-2.69,-1.3)--(-2.21,-1.1)--(-1.64,-0.8)--(-1.49,-0.7)--(-1.25,-0.5)--(-1.09,-0.3)--(-1,0)--(-1.09,0.3)--(-1.25,0.5)--(-1.49,0.7)--(-1.64,0.8)--(-2.21,1.1)--(-2.69,1.3)--(-3.89,1.7)--(-5,2)--(-3,2)--(-1.89,1.7)--(-0.69,1.3)--(-0.21,1.1)--(0.36,0.8)--(0.51,0.7)--(0.75,0.5)--(0.91,0.3);
\draw [dashed,color=blue,very thick] [label] plot file {ess1.dat};
\draw [color=blue,very thick] [label] plot file {ess2.dat};
\draw[] (0,-2) -- (0,2) node[above,very thick] {$\Im(\Lambda)$} coordinate(y axis);
\draw[] (-5,0) -- (2,0) 
node[right,very thick] {$\mathfrak{R}(\Lambda)$} coordinate( x axis);
\draw[color=red,very thick] (-5,0) -- (-3,0);
\draw[dashed,color=red,very thick] (-2,0) -- (-5,0);
\fill[red] (-3,0) circle (3pt);
\fill[red] (-2,0) circle (3pt);
\fill[blue] (-1,0) circle (3pt);
\fill[blue] (1,0) circle (3pt);
\node [black,above] at (-3.5,0) {$\sigma^-_{\text{abs}}$};
\node [black,above] at (-2.5,0) {$\sigma^+_{\text{abs}}$};
\node [black,below] at (1.15,0) {$\lambda$};
\node [black,below] at (-0.8,0) {$-\lambda$};
\node [black,below] at (-1.8,0) {$K_+$};
\node [black,below] at (-2.8,0) {$K_-$};
\node [black,below] at (-1.5,-2) {(a)};
\end{tikzpicture}
\hfill
\begin{tikzpicture}[x=1.26cm,y=1.33cm, scale=1]
\fill[green!20] (-4.75,-1.5)--(-3.94,-1.2)--(-3.31,-0.9)--(-2.86,-0.6)--(-2.59,-0.3)--(-2.5,0)--(-2.59,0.3)--(-2.86,0.6)--(-3.31,0.9)--(-3.94,1.2)--(-4.75,1.5)--(-2,1.5)--(-2,-1.5);
\draw [dashed,color=blue,very thick] (-2,-1.5)--(-2,1.5);
\draw [color=blue,very thick] plot file {ess3.dat};
\draw[] (0,-1.5) -- (0,1.5) node[above,very thick] {$\Im(\Lambda)$} coordinate(y axis);
\draw[] (-4.75,0) -- (0.25,0) 
node[right,very thick] {$\mathfrak{R}(\Lambda)$} coordinate( x axis);
\draw[color=red,very thick] (-4.75,0) -- (-3,0);
\draw[dashed,color=red,very thick] (-2,0) -- (-4.75,0);
\fill[red] (-3,0) circle (3pt);
\fill[red] (-2,0) circle (3pt);
\fill[blue] (-2.5,0) circle (3pt);
\node [black,above] at (-3.5,0) {$\sigma^-_{\text{abs}}$};
\node [black,above] at (-2.25,0) {$\sigma^+_{\text{abs}}$};
\node [black,below] at (-1.4,0) {$K_+^{\nu}=K_+$};
\node [black,below] at (-2.8,0) {$K_-$};
\node [black,below] at (-2.3,0) {$K_-^{\nu}$};
\node [black,below] at (-2.25,-1.5) {(b)};
\end{tikzpicture}
\caption{(a) shows the unweighted essential spectrum and the absolute spectrum of the linear operator $\mathcal{L}$ for $c>c^*$. The boundary of the unweighted essential spectrum is determined by the dispersion relations of $A_+$ (dashed blue curve) and $A_-$ (solid blue curve) and the green region is the interior of the unweighted essential spectrum.
The solid red line is the absolute spectrum $\sigma^+_{\text{abs}}$ (\ref{ab-}), while the dashed red line is the absolute spectrum $\sigma^+_{\text{abs}}$ (\ref{ab+}). (b)~shows that the unweighted essential spectrum is, for a weight $\nu=c/(2D(0))$ with $c\ge c^*$, shifted to the rightmost boundary of the absolute spectrum $\sigma^+_{\text{abs}}$. 
}
\label{essentialspectrum1}
\end{figure}
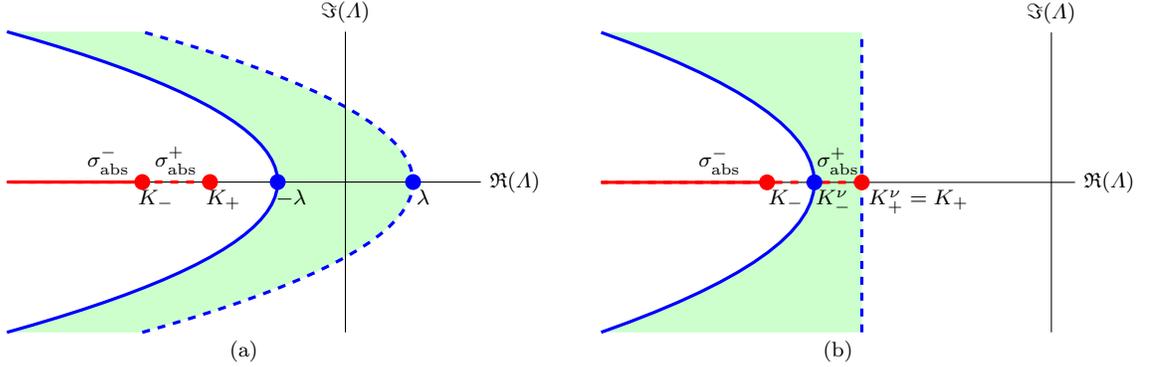

From (\ref{mu+}) we get that the absolute spectrum at $+\infty$ is given by
\begin{equation} \label{ab+}
    \sigma_{\text{abs}}^+=\left\{\Lambda\in\mathbb{R}\ \bigg|\ \Lambda<-\frac{c^2}{4D(0)}+R'(0) = -\frac{c^2}{4D_i}+\lambda =: K_+ \right\}.
\end{equation}
Similarly, from (\ref{mu-}) we get that the absolute spectrum at $-\infty$ is given by
\begin{equation} \label{ab-}
     \sigma_{\text{abs}}^-=\left\{\lambda\in\mathbb{R}\ \bigg|\ \Lambda<-\frac{c^2}{4D(1)}+R'(1) = -\frac{c^2}{4D_g}-\lambda =: K_-\right\}.
\end{equation}
That is, $\sigma_{\text{abs}}^-$ is always fully contained in the open left half plane including the origin, while $\sigma_{\text{abs}}^+$ is only fully contained in the open left half plane including the origin for $c\geq c^*=2\sqrt{\lambda D_i}$, see Figure \ref{essentialspectrum1}. 

The essential spectrum in the weighted space $\mathbb{H}^1_\nu(\mathbb{R})$ is determined by 
the operator 
\begin{equation} \nonumber
    \mathcal{T}^\nu(\Lambda)\left(\begin{aligned}
    &q\\
    &s
    \end{aligned}\right) :=\left(\frac{d}{dz}-\left(A(z;\Lambda)+\nu I\right)\right)\left(\begin{aligned}
    &q\\
    &s
    \end{aligned}\right) = 0\,,
\end{equation}
see \citep{kapitula2013spectral},
and the weighted asymptotic matrices are
\begin{equation}
\nonumber
    A_+^\nu(\Lambda)=A_+(\Lambda) + \nu I=
   \begin{pmatrix}
    \nu&1\\
    \dfrac{-R'(0)+\Lambda}{D(0)}&-\dfrac{c}{D(0)}+\nu
   \end{pmatrix},
\end{equation}
and
\begin{equation}
\nonumber
    A_-^\nu(\Lambda)=A_-(\Lambda) + \nu I=
     \begin{pmatrix}
    \nu&1\\
    \dfrac{-R'(1)+\Lambda}{D(1)}&-\dfrac{c}{D(1)}+\nu
   \end{pmatrix}.
\end{equation}
Hence, the boundary of the essential spectrum in the weighted space is given by the dispersion relations
\begin{equation} \nonumber
\begin{aligned}
&\Lambda_+^\nu=-D(0)k^2+i(c-2D(0)\nu)k+D(0)\nu^2-c\nu+R'(0),\\
&\Lambda_-^\nu=-D(1)k^2+i(c-2D(1)\nu)k+D(1)\nu^2-c\nu+R'(1).
\end{aligned}
\end{equation}
These dispersion relations still form two parabolas opening leftward and the intersections with the real axis now depend on $\nu$. 
We define the intersection of $\Lambda_+^\nu$ with the real axis as $K_+^{\nu}:=D(0)\nu^2-c\nu+R'(0)$, and the intersection of $\Lambda_-$ on the real axis as $K_-^{\nu}:=D(1)\nu^2-c\nu+R'(1)$. 
For $2\sqrt{D'(\beta)R(\beta)}<c<c^*$, $K_+^{\nu}$ is positive for all weights $\nu$, that is, $\Lambda_+^\nu$ always has a positive intersection on the real axis. 
In other words, for $2\sqrt{D'(\beta)R(\beta)}<c<c^*$ and in any weighted space $\mathbb{H}^1_\nu(\mathbb{R})$, parts of the boundary of the weighted essential spectrum lie in the open right half plane and the associated travelling wave solution is hence absolutely unstable.
For speed $c\ge c^*$, there exists a range of weights 
\begin{equation}
\label{RANGE}\nu\in\left(\frac{c-\sqrt{c^2-(c^*)^2}}{2D(0)},\frac{c+\sqrt{c^2+(c^*)^2}}{2D(0)}\right) \end{equation}
such that $K_+^\nu<0$, that is, $\Lambda_+$ has a negative intersection with the real axis. 
Furthermore, $K_-^\nu<K_+^\nu$. Therefore, for $c\ge c^*$, the unweighted essential spectrum is shifted into the open left half plane for weights in the above range \eqref{RANGE}. Furthermore, when $\nu=c/(2D(0))$, $K_+^\nu$ reaches its minimum, which coincides with $K_+$, the rightmost boundary of the absolute spectrum $\sigma^+_{abs}$~\eqref{ab+}. 
Note that $\nu=c/(2D(0))$ is the ideal one-sided weight \citep{davis2017absolute}, i.e. the weight that shifts the right most boundary of the essential spectrum furthest into the left half plane (since $\sigma^+_{abs}$ is to the right of $\sigma^-_{abs}$). See Figure~\ref{essentialspectrum1}. 

In conclusion, a travelling wave solution with speed $2\sqrt{D'(\beta)R(\beta)}<c<c^*$ is absolutely unstable and no weights exist to shift its unweighted essential spectrum into the open left half plane. In contrast, the absolute spectrum of a travelling wave solution with speed $ c\geq c^*$ is fully contained in the open left half plane including the origin and weights can be found that shift the unweighted essential spectrum into this region.

\section{Summary and future work}
\label{S:SUM}
\subsection{Summary of results}
We started this manuscript with a lattice-based discrete model introduced in \cite{johnston2017co} that explicitly accounts for  differences in individual and collective cell behaviour. Based on \cite{johnston2017co}, the discrete model has the continuous description (\ref{RDE_1}) obtained by using truncated Taylor series in the continuum limit. Our analysis focused on the case where $D_i>4D_g$ so that we can obtain a convex nonlinear diffusivity function $D(U)$, given by $(\ref{D(u)2})$, which changes sign twice in our domain of interest. Furthermore, the assumption of equal proliferation rates and zero death rates leads to a logistic kinetic term $R(U)$, given by (\ref{R(u)2}). The associated numerical simulations of (\ref{RDE_1}) with (\ref{D(u)2}) and (\ref{R(u)2}), see Figure \ref{travellingwavesolutionpicture1}, provided evidence of the existence of smooth monotone travelling wave solutions. To study these travelling wave solutions of (\ref{RDE_1}), we used a travelling wave coordinate $z=x-ct$ and looked for stationary solutions in the moving frame. Consequently, (\ref{RDE_1}) was transformed into the singular second-order ODE (\ref{ODE_1}) which we transformed into a singular system of first-order ODEs (\ref{ODEsystem_desingularised_0}). To remove the singularities, we used the stretched variable $D(u)d\xi=dz$ and transformed (\ref{ODEsystem_desingularised_0}) into system (\ref{ODEsystem_desingularised_1}). Next, we analysed the phase plane of the desingularised system (\ref{ODEsystem_desingularised_1}) and proved the existence of heteroclinic orbits connecting the equilibrium points $(0,0),(\alpha,0),(\beta,0)$ and $(1,0)$ for wave speeds $c\ge c^*$, given by (\ref{minimumwave}). Subsequently, based on the relation between the phase planes of (\ref{ODEsystem_desingularised_0}) and (\ref{ODEsystem_desingularised_1}), we proved the existence of a heteroclinic orbit in (\ref{ODEsystem_desingularised_0}) connecting the equilibrium points $(1,0)$ and $(0,0)$ passing through $(\alpha,0)$ and $(\beta,0)$, that are special points on the phase plane called a hole in the wall of singularities. That is, we proved the existence of smooth monotone travelling wave solutions of (\ref{RDE_1}) for $c\ge c^*$. In the end, we showed that the travelling wave solutions of (\ref{RDE_1}) with wave speeds $c<c^*$ are absolutely unstable, which in turn explained that the numerical simulations only provided travelling wave solutions with wave speeds $c\ge c^*$.

Based on our analysis, one-dimensional agent density profiles in the discrete model will eventually spread with a speed $c\ge c^*$ if the two types of agents have equal proliferation rates, zero death rates and different diffusivities satisfying $D_i>4D_g$. Notice that $c^*=2\sqrt{\lambda D_i}$, hence, the lowest speed for the travelling wave only relates to the diffusivity of individuals and is independent of the diffusivity of the gouped agents. That is, the diffusivity of grouped agents which is smaller than that of isolated agents ($D_i>4D_g$) does not give restrictions for the lowest speed of the moving front. Consequently, we infer that the speed of invasion processes for organisms, for instance, cells, is mainly determined by the behaviour of individuals. Furthermore, the Fisher-KPP equation also has a minimum wave speed for the existence of smooth monotone travelling wave solutions \citep{ararticle,fife2013mathematical}. Hence, a discrete mechanism of invasion processes considering the differences in individual and collective behaviours can lead to a macroscopic behaviour similar to that observed in the discrete mechanism with no differences in isolated and grouped agents. 
\subsection{Smooth travelling wave solutions for positive $D(U)$}\label{SS:POS}
If $D_i<4D_g$, then the nonlinear diffusivity function $D(U)$ is positive for $U\in[0,1]$, see Figure \ref{appen2}a. Thus the corresponding system of first-order ODEs (\ref{ODEsystem_desingularised_0}) is not singular, and the nullcline $p=-D(u)R(u)/c$ does not cross $u$-axis, see Figure \ref{appen2}b.
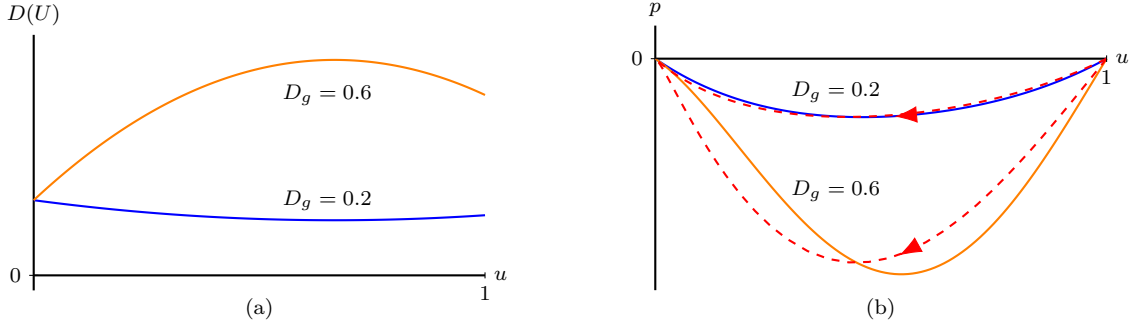
\begin{figure}
\centering
\begin{tikzpicture}[x=30cm,y=20cm, scale=0.2]
\draw[thick] (0,-0.05) -- (0,0.8) node[above,very thick] {$D(U)$} coordinate( y axis);
\draw[thick] (0,0) -- (1,0) 
node[right,very thick] {$u$} coordinate( x axis);
\draw [color=blue,thick] [label] plot file {D_Dg2.dat};
\draw [color=orange,thick] [label] plot file {D_Dg3.dat};
\foreach \pos in {1}
\draw[shift={(\pos,0)}] (0,0.01) -- (0,-0.01) node[below] {$\pos$};
\foreach \pos in {0}  
\draw[shift={(0,\pos)}] (0.01,0) -- (-0.01,0) node[left] {$\pos$};
\node [black] at (0.65,0.25) {$D_g=0.2$};
\node [black] at (0.65,0.6) {$D_g=0.6$};
\node [black,below] at (0.5,-0.05) {(a)};
\end{tikzpicture}
\hfill
\begin{tikzpicture}[x=30cm,y=110cm, scale=0.2]
\draw[thick] (0,-0.14) -- (0,0.02) node[above,very thick] {$p$} coordinate( y axis);
\draw[thick] (0,0) -- (1,0) 
node[right,very thick] {$u$} coordinate( x axis);
\draw [color=blue,thick] [label] plot file {D_Dg_p2.dat};
\draw [color=orange,thick] [label] plot file {D_Dg_p3.dat};
\draw [dashed,color=red,thick] [label] plot file {D_Dg_p2_orbit.dat};
\draw [dashed,color=red,thick] [label] plot file {D_Dg_p3_orbit.dat};
\draw[color=red,-{Latex[length=3mm]}] (0.55,-0.0345) -- (0.53,-0.035);
\draw[color=red,-{Latex[length=3mm]}] (0.56,-0.116) -- (0.54,-0.119);
\foreach \pos in {1}
\draw[shift={(\pos,0)}] (0,0.002) -- (0,-0.002) node[below] {$\pos$};
\foreach \pos in {0}  
\draw[shift={(0,\pos)}] (0.01,0) -- (-0.01,0) node[left] {$\pos$};
\node [black] at (0.4,-0.02) {$D_g=0.2$};
\node [black] at (0.4,-0.08) {$D_g=0.6$};
\node [black,below] at (0.5,-0.14) {(b)};
\end{tikzpicture}
\caption{(a) shows $D(U)$ with $D_i=0.25$ and two different $D_g$. (b) gives the corresponding phase planes of system (\ref{ODEsystem_desingularised_0}) for $\lambda=0.75$, $c=1$, $D_i=0.25$, $D_g=0.2$ and $D_g=0.6$, respectively. The two solid curves are the nullclines $p=-D(u)R(u)/c$ with $D_g=0.2$ (blue curve) and $D_g=0.6$ (orange curve), respectively. The red dashed lines are the corresponding heteroclinic orbits representing travelling wave solutions in (\ref{RDE_1}).}
\label{appen2}
\end{figure}
In other words, $(0,0)$ and $(1,0)$ are the only equilibrium points. Following the same method as applied in Section~\ref{S:EXIST}.2, we obtain the lower bound
\begin{equation}
\nonumber
   S_1=\sup_{u\in(0,1)}2\sqrt{\frac{D(u)R(u)}{u}}=\sup_{u\in(0,1)}2\sqrt{\lambda(1-u)D(u)},
\end{equation}
such that there exist smooth monotone travelling wave solutions of (\ref{RDE_1}) for $c\ge S_1$. The origin is still a stable node for $c\ge2\sqrt{\lambda D_i}:=S_2$ and $S_1\ge S_2$. So, if $S_1\ne S_2$, $c\ge S_1$ is only a sufficient condition because there may exist smooth monotone travelling wave solutions of (\ref{RDE_1}) for wave speeds $S_2\le c< S_1$. Thus, we can only conclude that the minimum wave speed is in the range
\begin{equation}
\label{condition_threshold1}
  S_2\le \hat{c}\le S_1,
\end{equation}
such that there exist smooth monotone nonnegative travelling wave solutions of (\ref{RDE_1}) for $c\ge \hat{c}$. Note that the minimum wave speed $\hat{c}$ can be different from the minimum wave speed $c^*$ in Theorem~\ref{THEOREM}, and Lemma~\ref{L:THRES} does not necessarily hold.

This estimate is consistent with the result in \citet{malaguti2003sharp} obtained by using the \emph{comparison method} introduced by \citet{aronson1978multidimensional}. The corresponding numerical simulations also give the expected results, see Figure \ref{travellingwavesolutionpicture2}. \citet{witelski1994asymptotic} obtained an asymptotic travelling wave solution for a PDE motivated by polymer diffusion with a positive nonlinear diffusivity function and logistic kinetics for wave speeds greater than a minimum wave speed which is greater than $S_2$. This is consistent with the estimate of the minimum wave speed in (\ref{condition_threshold1}). For solutions with an asymptotic wave speed equal to $S_2$, the front of the travelling wave is called a \emph{pulled front}; for solutions with asymptotic speeds greater than $S_2$, the front of the travelling wave is called a \emph{pushed front} \citep{van2003front}. Unravelling the differences in wave speed selection remains to be explored.
\begin{figure}
\centering
\begin{tikzpicture}[x=4cm,y=2cm, scale=1]
\draw[thick] (0,0) -- (0,2) node[above,very thick] {$c$} coordinate(y axis);
\draw[thick] (0,0) -- (1.4,0) 
node[right,very thick] {$\eta$} coordinate( x axis);
\foreach \pos in {0,0.5,1,1.4}  
\draw[shift={(\pos,0)}] (0,0.02) -- (0,-0.02) node[below] {$\pos$};
\draw [color=orange,very thick] [label] plot file {speedwithICs3.dat};
\draw[thick,dashed,red] (0,0.866) -- (1.4,0.866);
\draw[thick,dashed,red] (0,1.1) -- (1.4,1.1);
\node [black,right] at (0.6,0.7) {$S_2=0.866$};
\node [black,right] at (0.6,1.26) {$S_1=1.1$};
\node [black,right] at (1,1.8) {$D_g=0.6$};
\node [black,below] at (0.75,0) {(a)};
\end{tikzpicture}
\hfill
\begin{tikzpicture}[x=4cm,y=2cm, scale=1]
\draw[thick] (0,0) -- (0,2) node[above,very thick] {$c$} coordinate(y axis);
\draw[thick] (0,0) -- (1.4,0) 
node[right,very thick] {$\eta$} coordinate( x axis);
\foreach \pos in {0,0.5,1,1.4}  
\draw[shift={(\pos,0)}] (0,0.02) -- (0,-0.02) node[below] {$\pos$};
\draw [color=blue,very thick] [label] plot file {speedwithICs2.dat};
\draw[thick,dashed,red] (0,0.866) -- (1.4,0.866);
\draw[thick,dashed,red] (0,1.1) -- (1.4,1.1);
\node [black,right] at (0.6,0.7) {$S_2=0.866$};
\node [black,right] at (0.6,1.26) {$S_1=1.1$};
\node [black,right] at (1,1.8) {$D_g=0.2$};
\node [black,below] at (0.75,0) {(b)};
\end{tikzpicture}
\caption{(a) gives the wave speed as a function of the initial condition $U(x,0)=1/2+\text{tanh}\left(-\eta (x-40)\right)/2$. Notice that as $\eta$ grows to infinity this initial condition limits to the Heaviside initial condition. Parameters are $\lambda=0.75$, $D_i=0.25$ and $D_g=0.6$. The wave speed reaches its minimum which is between $S_1$ and $S_2$ and then converges to a bigger value which is still between $S_1$ and $S_2$. In (b), $D_g=0.2$ while the other parameters are the same as in (a). In this case, the wave speed converges to $S_2$.}
\label{travellingwavesolutionpicture2}
\end{figure}
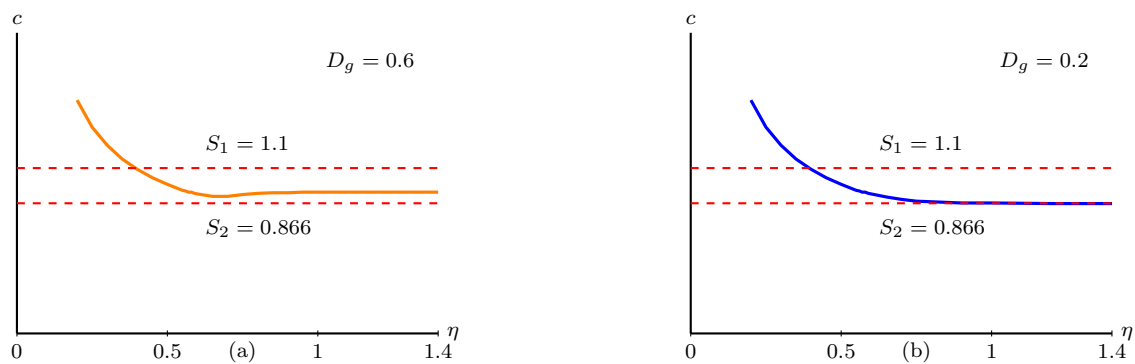
\subsection{Shock-fronted travelling waves} \label{SS:shocks}
In Section~\ref{S:EXIST}, we mainly considered the equilibrium point $(0,0)$ as a stable node in the phase plane of system (\ref{ODEsystem_desingularised_1}). With $(0,0)$ a stable node, $(\beta,0)$ is also a stable node based on (\ref{relation1}). However, (\ref{relation1}) does not hold for any convex $D(U)$ which changes sign twice. For instance, for
\begin{equation}
\label{D(u)3}
    \hat{D}(U)=(U-0.1)(U-0.3),
\end{equation}
condition (\ref{condition2}) and condition (\ref{condition1}) become 
\begin{equation}
    \nonumber
    c\ge 2\sqrt{\hat{D}(0)R'(0)}=0.3,\quad
    c\ge2\sqrt{\hat{D}'(0.3)R(0.3)}\approx0.355.
\end{equation}
With the nonlinear diffusivity function $\hat{D}(U)$, the equilibrium point $(0,0)$ is a stable node and the equilibrium point $(\beta,0)$ is a stable spiral for speeds $0.3<c<0.355$ in (\ref{ODEsystem_desingularised_1}).
In this case, only shock-fronted travelling wave solutions of (\ref{RDE_1}) can exist since (\ref{ODEsystem_desingularised_1}) no longer possesses heteroclinic orbits connecting to $(\beta,0)$ that do not cross the walls of singularities, see Figure \ref{phaseplane_pic5}. 
\begin{figure}
\centering
\begin{tikzpicture}[x=6cm,y=25cm, scale=1]
\fill[green!20] (0.1,0.01)--(0.3,0.01)--(0.3,-0.15)--(0.1,-0.15);
\draw[dashed,red,thick] (0.1,-0.15) -- (0.1,0.01); 
\draw[dashed,red,thick] (0.3,-0.15) -- (0.3,0.01); 
\draw[thick] (0,-0.15) -- (0,0.01) node[above,thick] {$p$} coordinate(y axis);
\draw[blue,thick] (0,0) -- (1,0) 
node[black,right,thick] {$u$} coordinate( x axis);
\draw [color=blue,thick] [label] plot file {Heteroclinicorbit8.dat};
\draw [color=red,very thick] [label] plot file {Heteroclinicorbit5.dat};
\node[black,below] at (-0.03,0) {0};
\node[black,below] at (1.03,0) {1};
\node[black,below] at (0.14,0) {$0.1$};
\node[black,below] at (0.25,0) {$0.3$};
\fill[black] (0,0) circle (2pt);
\fill[black] (0.1,0) circle (2pt);
\fill[black] (0.3,0) circle (2pt);
\fill[black] (1,0) circle (2pt);
\draw[color=red,-{Latex[length=3mm]}] (0.5,-0.089) -- (0.48,-0.086);
\draw[color=red,-{Latex[length=3mm]}] (0.248,-0.028) -- (0.24,-0.025);
\node [black,below] at (0.5,-0.15) {(a)};
\end{tikzpicture}
\hfill
\begin{tikzpicture}[x=6cm,y=25cm, scale=1]
\fill[green!20] (0.1,0.01)--(0.3,0.01)--(0.3,-0.15)--(0.1,-0.15);
\draw[dashed,red,thick] (0.1,-0.15) -- (0.1,0.01); 
\draw[dashed,red,thick] (0.3,-0.15) -- (0.3,0.01); 
\draw[thick] (0,-0.15) -- (0,0.01) node[above,thick] {$p$} coordinate(y axis);
\draw[blue,thick] (0,0) -- (1,0) 
node[black,right,thick] {$u$} coordinate( x axis);
\draw [color=blue,thick] [label] plot file {Heteroclinicorbit8.dat};
\draw [color=red,very thick] [label] plot file {Heteroclinicorbit5.dat};
\node[black,below] at (-0.03,0) {0};
\node[black,below] at (1.03,0) {1};
\node[black,below] at (0.14,0) {$0.1$};
\node[black,below] at (0.25,0) {$0.3$};
\fill[black] (0,0) circle (2pt);
\fill[black] (0.1,0) circle (2pt);
\fill[black] (0.3,0) circle (2pt);
\fill[black] (1,0) circle (2pt);
\draw[color=red,-{Latex[length=3mm]}] (0.5,-0.089) -- (0.48,-0.086);
\draw[color=red,-{Latex[length=3mm]}] (0.24,-0.0255) -- (0.248,-0.028);
\node [black,below] at (0.5,-0.15) {(b)};
\end{tikzpicture}
\caption{(a) shows the phase plane of the desingularised system (\ref{ODEsystem_desingularised_1}) with $\hat{D}(u)$, $c=0.3$ and $\lambda=0.75$. The vertical dashed lines are the wall of singularities at $u=0.1$ and $u=0.3$. The blue lines are the nullclines $p=0$ and $p=-D(u)R(u)/c$. The red line is the heteroclinic orbit connecting $(1,0)$ to $(0.3,0)$. (b) shows the phase plane of system (\ref{ODEsystem_desingularised_0}) with $\hat{D}(u)$, $c=0.3$ and $\lambda=0.75$. The vertical dashed lines are the walls of singularities $u=0.1$ and $u=0.3$. The blue lines are the nullclines $p=0$ and $p=-D(u)R(u)/c$. The red line shows the orientation of the same trajectory in (a) on different sides of the wall of singularities $u=0.3$.}
\label{phaseplane_pic5}
\end{figure}
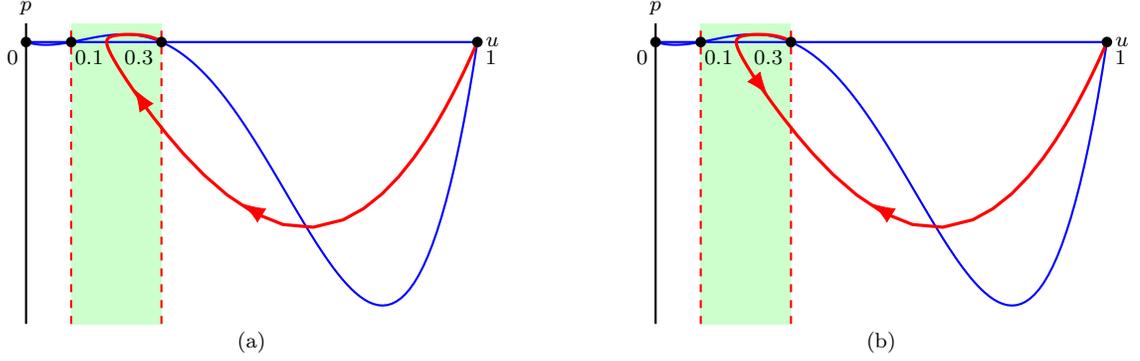
The corresponding numerical simulation of (\ref{RDE_1}) indeed gives a shock-fronted travelling wave solution with a speed $c=0.3$, see Figure \ref{phaseplane_pic4}. 

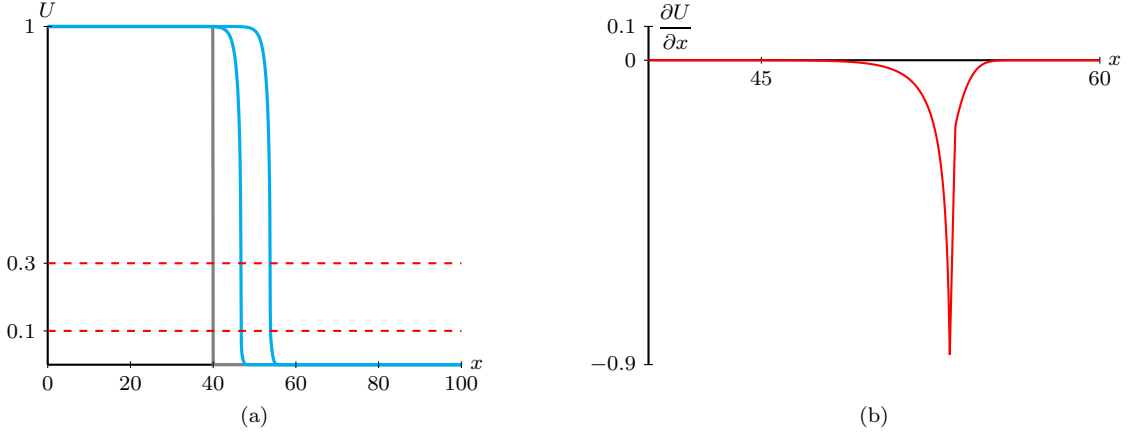
\begin{figure}
\centering
\begin{tikzpicture}[x=0.055cm,y=4.5cm, scale=1]
\draw[thick] (0,0) -- (0,1) node[above,very thick] {$U$} coordinate(y axis);
\draw[thick] (0,0) -- (100,0) 
node[right,very thick] {$x$} coordinate( x axis);
\draw [color=gray,very thick] [label] plot file {travellingwavesolution_1.dat};
\draw [color=cyan,very thick] [label] plot file {travellingwavesolution_6.dat};
\draw [color=cyan,very thick] [label] plot file {travellingwavesolution_7.dat};
\foreach \pos in {0,20,40,60,80,100}
\draw[shift={(\pos,0)}] (0,0.01) -- (0,-0.01) node[below] {$\pos$};
\foreach \pos in {0.1,0.3,1}  
\draw[shift={(0,\pos)}] (1,0) -- (-1,0) node[left] {$\pos$};
\draw[thick,dashed,red] (0,0.3) -- (100,0.3);
\draw[thick,dashed,red] (0,0.1) -- (100,0.1);
\node [black,below] at (50,-0.1) {(a)};
\end{tikzpicture}
\hfill
\begin{tikzpicture}[x=0.3cm,y=4.5cm, scale=1]
\draw[thick] (40,-0.9) -- (40,0.1) node[right,thick] {$\dfrac{\partial U}{\partial x}$} coordinate(y axis);
\draw[thick] (40,0) -- (60,0) 
node[black,right,thick] {$x$} coordinate( x axis);
\foreach \pos in {45,60}
\draw[shift={(\pos,0)}] (0,0.01) -- (0,-0.01) node[below] {$\pos$};
\foreach \pos in {0.1,0,-0.9}
\draw[shift={(40,\pos)}] (0.2,0) -- (-0.2,0) node[left] {$\pos$};
\draw [color=red,thick] [label] plot file {pplaneshock.dat};
\node [black,below] at (50,-1) {(b)};
\end{tikzpicture}
\hfill
\caption{(a) shows the evolution of a Heaviside initial condition to a smooth travelling wave solution obtained by simulating (\ref{RDE_1}) with (\ref{D(u)3}) and (\ref{R(u)2}) with $\lambda=0.75$ at $t=0$, $t=25$ and $t=50$. Notice that $D(U)=0$ at $\alpha=0.1$ and $\beta=0.3$. The travelling wave solution eventually has a constant positive speed, $c=0.3$. (b) shows $\partial U/\partial x$  corresponding to the numerical solution in (a) at $t=50$ and for $x$ between $40$ and $60$.}
\label{phaseplane_pic4}
\end{figure}

It is not a surprise to see shock-fronted travelling wave solutions in negative nonlinear diffusion equations. Shocks in negative nonlinear diffusion equations with no kinetic terms have been studied in the context of many physical phenomena, such as the movement of moisture in partially saturated porous media \citep{dicarlo2008nonmonotonic}; the motion of nanofluids \citep{landman2011terraced} and these kinds of PDEs also arise in the study of Cahn-Hilliard models \citep{witelski1995shocks}. Numerical simulations of (\ref{RDE_1}) with nonlinear diffusivity function $(\ref{D(u)2})$ and Allee kinetics \eqref{R(u)1} also lead to shock-fronted solutions, see \citet{johnston2017co}. In addition, Allee kinetics support shock-fronted travelling wave solutions for reaction-diffusion-advection equations with small diffusion coefficients \citep{sewalt2016influences,wang2018persistence}. The analysis of shock-fronted travelling wave solutions in nonlinear diffusion-reaction equations with generic diffusivity functions and logistic kinetics is left for future work.

\subsection{Point spectrum}
To fully establish spectral stability of the operator $\mathcal{L}$ \eqref{L}, we also need to determine the point spectrum of $\mathcal{L}$ and show that it is contained in the open left half plane including the origin when $c\geq c_*$ provided our perturbations stay in an appropriately chosen Hilbert space $\mathcal{X}$. With this in mind, we define 
\begin{equation}\label{eq:weight}
w(z) := D(\hat{u})q(z) e^{\int c/(2D(\hat{u}(t)))dt}.
\end{equation}
Then if $\mathcal{L} q = \Lambda q$ \eqref{EIG} we have that $w$ will solve 
$$
\mathcal{M}w(z) := D(\hat{u}) w_{zz} +  \left( R'(\hat{u}) - \frac{c\left(c+ 2D'(\hat{u})\hat{u}_z \right)}{4D(\hat{u})}
\right)w(z) = \Lambda w(z). 
$$
We have thus reduced the problem to showing that $\mathcal{M}$ is negative semi-definite on some appropriately chosen Hilbert space $\mathcal{X}$. Unfortunately, the natural choice for such a Hilbert space in these problems is the one with ``inner product" 
$$
(u,v) := \int \frac{u v }{D(\hat{u})} dz, 
$$
but the sign change in $D(\hat{u})$ means that this is actually no longer an inner product (it is strictly negative for a localised pulse near where $D(\hat{u})$ is negative for instance).

However, if we instead work with the desingularised system \eqref{ODEsystem_desingularised_1}, then for a perturbation $\tilde{q}(\xi)$ about one of the three heteroclinic orbits $\tilde{u}$, linearising gives the eigenvalue problem for the linearised desingularised system
\begin{equation}\label{eq:desinglin}
\tilde{q}_{\xi \xi} + c \tilde{q}_{\xi} + F(\tilde{u}) 
\tilde{q} = \Lambda \tilde{q}\,,
\end{equation}
where $F(\tilde{u})$ is defined in \eqref{XJ}.
The standard Liouville transformation $ \tilde{w}(\xi):= \tilde{q}(\xi) e^{c\xi/2}$ now does lead to a self adjoint eigenvalue problem in terms of $\tilde{w}(\xi)$ 
\begin{equation*} 
\tilde{w}_{\xi\xi} +\left(
F(\tilde{u})
- \frac{c^2}{4}\right)\tilde{w}(\xi) = \Lambda \tilde{w}(\xi) \,.
\end{equation*}
Here, one can show explicitly that the operator 
$$
\tilde{\mathcal{M}} := \frac{d^2}{d\xi^2} + \left(F(\tilde{u})- \frac{c^2}{4}\right)
$$
is negative semi-definite precisely when $c\geq c_*$. Indeed, as we are assuming that $D_i>4D_g$, the potential term in $\tilde{\mathcal{M}}$ satisfies 
\begin{equation*}
 \left(
 F(\tilde{u})- \frac{c^2}{4}\right) < \frac{1}{4}\left( -c^2 + \lambda D_i(4 -32 \tilde{u} + 63\tilde{u}^2-36\tilde{u}^3 ) \right)
  \end{equation*}
 and the polynomial term 
 $
 4 -32 \tilde{u} + 63\tilde{u}^2-36\tilde{u}^3 
 $
 has a maximum value of $4$ when $\tilde{u} \in [0,1]$ (at $\tilde{u} = 0$). So, we have that 
 $
F(\tilde{u})- c^2/4 \leq 0 
 $
 when $c\geq c_* = 2 \sqrt{\lambda D_i}$. Thus, $\tilde{\mathcal{M}}$ is a negative semi-definite operator in the space of perturbations which decay faster than $e^{c\xi}$, that is,  $\mathbb{H}^1_{c}$. This is usually referred to as a transient instability in the stability literature \citep{sandstede2002stability, sherratt2014mathematical}.
 
Given what was just shown, the only remaining step in the proof of stability of these travelling wave solutions for $c\geq c_*$ is how to relate the eigenvalue problem of the desingularised system \eqref{eq:desinglin} for the three different heteroclinic orbits to the spectrum of the operator $\mathcal{L}$ for $\hat{u}$. Due to the singular nature of the operator, it is unclear how to even define the ``natural'' Hilbert spaces which should act as domains for the original linearised problem. Further, the weighting given in \eqref{eq:weight} involves a nonlinear, singular exponential weight, and to the best of our knowledge there is no such work which describes the dynamic effects of stability or instability in these cases. So, we cannot even say whether we would have only a transient instability even if we could show that the ``natural'' operator was negative definite on an appropriate domain.

\begin{acknowledgements}
The authors would like to thank PN Davis and M Wechselberger for fruitful discussions. We also thank the two referees for their helpful suggestions.
\end{acknowledgements}

\newpage
\bibliographystyle{spbasic}

\begin{thebibliography}{48}
\providecommand{\natexlab}[1]{#1}
\providecommand{\url}[1]{{#1}}
\providecommand{\urlprefix}{URL }
\expandafter\ifx\csname urlstyle\endcsname\relax
  \providecommand{\doi}[1]{DOI~\discretionary{}{}{}#1}\else
  \providecommand{\doi}{DOI~\discretionary{}{}{}\begingroup
  \urlstyle{rm}\Url}\fi
\providecommand{\eprint}[2][]{\url{#2}}

\bibitem[{Allee and Bowen(1932)}]{allee1932studies}
Allee W, Bowen ES (1932) Studies in animal aggregations: {M}ass protection
  against colloidal silver among goldfishes. Journal of {E}xperimental
  {Z}oology 61(2):185--207

\bibitem[{Anguige and Schmeiser(2009)}]{anguige2009one}
Anguige K, Schmeiser C (2009) A one-dimensional model of cell diffusion and
  aggregation, incorporating volume filling and cell-to-cell adhesion. Journal
  of Mathematical Biology 58(3):395

\bibitem[{Aronson(1980)}]{aronson1980density}
Aronson DG (1980) Density-dependent interaction-diffusion systems. In: Dynamics
  and Modelling of Reactive Systems, Elsevier, pp 161--176

\bibitem[{Aronson and Weinberger(1978)}]{aronson1978multidimensional}
Aronson DG, Weinberger HF (1978) Multidimensional nonlinear diffusion arising
  in population genetics. Advances in Mathematics 30(1):33--76

\bibitem[{Barenblatt et~al.(1993)Barenblatt, Bertsch, Passo, and
  Ughi}]{barenblatt1993degenerate}
Barenblatt G, Bertsch M, Passo RD, Ughi M (1993) A degenerate pseudoparabolic
  regularization of a nonlinear forward-backward heat equation arising in the
  theory of heat and mass exchange in stably stratified turbulent shear flow.
  SIAM Journal on Mathematical Analysis 24(6):1414--1439

\bibitem[{Bramson et~al.(1986)Bramson, Calderoni, De~Masi, Ferrari, Lebowitz,
  and Schonmann}]{bramson1986microscopic}
Bramson M, Calderoni P, De~Masi A, Ferrari P, Lebowitz J, Schonmann RH (1986)
  Microscopic selection principle for a diffusion-reaction equation. Journal of
  Statistical Physics 45(5-6):905--920

\bibitem[{Codling et~al.(2008)Codling, Plank, and Benhamou}]{codling2008random}
Codling EA, Plank MJ, Benhamou S (2008) Random walk models in biology. Journal
  of the Royal Society Interface 5(25):813--834

\bibitem[{Courchamp et~al.(1999)Courchamp, Clutton-Brock, and
  Grenfell}]{courchamp1999inverse}
Courchamp F, Clutton-Brock T, Grenfell B (1999) Inverse density dependence and
  the {A}llee effect. Trends in Ecology \& Evolution 14(10):405--410

\bibitem[{Davis et~al.(2017)Davis, van Heijster, and
  Marangell}]{davis2017absolute}
Davis PN, van Heijster P, Marangell R (2017) Absolute instabilities of
  travelling wave solutions in a {K}eller--{S}egel model. Nonlinearity
  30(11):4029

\bibitem[{Davis et~al.(2019)Davis, van Heijster, and
  Marangell}]{davis2018spectral}
Davis PN, van Heijster P, Marangell R (2019) Spectral stability of travelling
  wave solutions in a {K}eller--{S}egel model. Applied Numerical Mathematics
  141:54--61

\bibitem[{Deroulers et~al.(2009)Deroulers, Aubert, Badoual, and
  Grammaticos}]{deroulers2009modeling}
Deroulers C, Aubert M, Badoual M, Grammaticos B (2009) Modeling tumor cell
  migration: {F}rom microscopic to macroscopic models. Physical Review E
  79(3):{031}{917}

\bibitem[{DiCarlo et~al.(2008)DiCarlo, Juanes, LaForce, and
  Witelski}]{dicarlo2008nonmonotonic}
DiCarlo DA, Juanes R, LaForce T, Witelski TP (2008) Nonmonotonic traveling wave
  solutions of infiltration into porous media. Water Resources Research
  44(2):{W02}{406}

\bibitem[{Druckenbrod and Epstein(2007)}]{druckenbrod2007behavior}
Druckenbrod NR, Epstein ML (2007) Behavior of enteric neural crest-derived
  cells varies with respect to the migratory wavefront. Developmental Dynamics
  236(1):84--92

\bibitem[{Ferracuti et~al.(2009)Ferracuti, Marcelli, and
  Papalini}]{ferracuti2009travelling}
Ferracuti L, Marcelli C, Papalini F (2009) Travelling waves in some
  reaction-diffusion-aggregation models. Advances in Dynamical Systems and
  Applications 4(1):19--33

\bibitem[{Fife(2013)}]{fife2013mathematical}
Fife PC (2013) Mathematical {A}spects of {R}eacting and {D}iffusing {S}ystems,
  vol~28. Springer Science \& Business Media

\bibitem[{Fisher(1937)}]{fisher1937wave}
Fisher RA (1937) The wave of advance of advantageous genes. Annals of Eugenics
  7(4):355--369

\bibitem[{Harley et~al.(2014{\natexlab{a}})Harley, van Heijster, Marangell,
  Pettet, and Wechselberger}]{harley2014existence}
Harley K, van Heijster P, Marangell R, Pettet GJ, Wechselberger M
  (2014{\natexlab{a}}) Existence of traveling wave solutions for a model of
  tumor invasion. SIAM Journal on Applied Dynamical Systems 13(1):366--396

\bibitem[{Harley et~al.(2014{\natexlab{b}})Harley, van Heijster, Marangell,
  Pettet, and Wechselberger}]{harley2014novel}
Harley K, van Heijster P, Marangell R, Pettet GJ, Wechselberger M
  (2014{\natexlab{b}}) Novel solutions for a model of wound healing
  angiogenesis. Nonlinearity 27(12):2{975}

\bibitem[{Harley et~al.(2015)Harley, van Heijster, Marangell, Pettet, and
  Wechselberger}]{harley2015numerical}
Harley K, van Heijster P, Marangell R, Pettet GJ, Wechselberger M (2015)
  Numerical computation of an {E}vans function for travelling waves.
  Mathematical Biosciences 266:36--51

\bibitem[{H{\"o}llig(1983)}]{hollig1983existence}
H{\"o}llig K (1983) Existence of infinitely many solutions for a forward
  backward heat equation. Transactions of the American Mathematical Society
  278(1):299--316

\bibitem[{Johnston et~al.(2012)Johnston, Simpson, and Baker}]{johnston2012mean}
Johnston ST, Simpson MJ, Baker RE (2012) Mean-field descriptions of collective
  migration with strong adhesion. Physical Review E 85(5):{051}{922}

\bibitem[{Johnston et~al.(2017)Johnston, Baker, McElwain, and
  Simpson}]{johnston2017co}
Johnston ST, Baker RE, McElwain DLS, Simpson MJ (2017) Co-operation,
  competition and crowding: a discrete framework linking {A}llee kinetics,
  nonlinear diffusion, shocks and sharp-fronted travelling waves. Scientific
  Reports 7:{42}{134}

\bibitem[{Jones(1995)}]{jones1995geometric}
Jones CK (1995) Geometric singular perturbation theory. In: Johnson R (ed)
  Dynamical Systems: Lectures Given at the 2nd Session of the Centro
  Internazionale Matematico Estivo (C.I.M.E.) held in Montecatini Terme, Italy,
  June 13--22, 1994, Springer Berlin Heidelberg, Berlin, Heidelberg, pp 44--118

\bibitem[{Jordan and Smith(1999)}]{jordan1999nonlinear}
Jordan DW, Smith P (1999) Nonlinear {O}rdinary {D}ifferential {E}quations: {A}n
  {I}ntroduction to {D}ynamical {S}ystems, vol~2. Oxford University Press, USA

\bibitem[{Kapitula and Promislow(2013)}]{kapitula2013spectral}
Kapitula T, Promislow K (2013) Spectral and dynamical stability of nonlinear
  waves. Springer

\bibitem[{Khain et~al.(2007)Khain, Sander, and Schneider-Mizell}]{Khain2007}
Khain E, Sander LM, Schneider-Mizell CM (2007) The role of cell-cell adhesion
  in wound healing. Journal of Statistical Physics 128(1-2):209--218

\bibitem[{Khain et~al.(2011)Khain, Katakowski, Hopkins, Szalad, Zheng, Jiang,
  and Chopp}]{khain2011collective}
Khain E, Katakowski M, Hopkins S, Szalad A, Zheng X, Jiang F, Chopp M (2011)
  Collective behavior of brain tumor cells: {T}he role of hypoxia. Physical
  Review E 83(3):{031}{920}

\bibitem[{Kolmogorov et~al.(1937)Kolmogorov, Petrovsky, and
  Piscounov}]{ararticle}
Kolmogorov A, Petrovsky I, Piscounov N (1937) {\'E}tude de l'{\'e}quation de la
  diffusion avec croissance de la quantit{\'e} de mati{\`e}re et son
  application {\`a} un probl{\`e}me biologique. Moscow University Mathematics
  Bulletin 1:1--25

\bibitem[{Landman and White(2011)}]{landman2011terraced}
Landman KA, White LR (2011) Terraced spreading of nanofilms under a
  nonmonotonic disjoining pressure. Physics of Fluids 23(1):{012}{004}

\bibitem[{Larson(1978)}]{larson1978transient}
Larson DA (1978) Transient bounds and time-asymptotic behavior of solutions to
  nonlinear equations of {F}isher type. SIAM Journal on Applied Mathematics
  34(1):93--104

\bibitem[{Mack et~al.(2000)Mack, Simberloff, Mark~Lonsdale, Evans, Clout, and
  Bazzaz}]{mack2000biotic}
Mack RN, Simberloff D, Mark~Lonsdale W, Evans H, Clout M, Bazzaz FA (2000)
  Biotic invasions: Causes, epidemiology, global consequences, and control.
  Ecological Applications 10(3):689--710

\bibitem[{Maini et~al.(2006)Maini, Malaguti, Marcelli, and
  Matucci}]{maini2006diffusion}
Maini PK, Malaguti L, Marcelli C, Matucci S (2006) Diffusion-aggregation
  processes with mono-stable reaction terms. Discrete and Continuous Dynamical
  Systems Series B 6(5):1175--1189

\bibitem[{Maini et~al.(2007)Maini, Malaguti, Marcelli, and
  Matucci}]{maini2007aggregative}
Maini PK, Malaguti L, Marcelli C, Matucci S (2007) Aggregative movement and
  front propagation for bi-stable population models. Mathematical Models and
  Methods in Applied Sciences 17(9):1351--1368

\bibitem[{Malaguti and Marcelli(2003)}]{malaguti2003sharp}
Malaguti L, Marcelli C (2003) Sharp profiles in degenerate and doubly
  degenerate {F}isher-{K}pp equations. Journal of Differential Equations
  195(2):471--496

\bibitem[{Murray(2002)}]{murray2002mathematical}
Murray JD (2002) Mathematical {B}iology: {I}. {A}n {I}ntroduction. Mathematical
  Biology, Springer

\bibitem[{Novick-Cohen and Pego(1991)}]{novick1991stable}
Novick-Cohen A, Pego RL (1991) Stable patterns in a viscous diffusion equation.
  Transactions of the American Mathematical Society 324(1):331--351,
  \doi{10.1090/S0002-9947-1991-1015926-7}

\bibitem[{Pego and Penrose(1989)}]{pego1989front}
Pego RL, Penrose O (1989) Front migration in the nonlinear cahn-hilliard
  equation. Proceedings of the Royal Society of London A Mathematical and
  Physical Sciences 422(1863):261--278

\bibitem[{Perona and Malik(1990)}]{perona1990scale}
Perona P, Malik J (1990) Scale-space and edge detection using anisotropic
  diffusion. IEEE Transactions on Pattern Analysis and Machine Intelligence
  12(7):629--639

\bibitem[{Pettet et~al.(2000)Pettet, McElwain, and Norbury}]{pettet2000lotka}
Pettet GJ, McElwain DLS, Norbury J (2000) Lotka-{V}olterra equations with
  chemotaxis: {W}alls, barriers and travelling waves. Mathematical Medicine and
  Biology: A Journal of the IMA 17(4):395--413

\bibitem[{Poujade et~al.(2007)Poujade, Grasland-Mongrain, Hertzog, Jouanneau,
  Chavrier, Ladoux, Buguin, and Silberzan}]{poujade2007collective}
Poujade M, Grasland-Mongrain E, Hertzog A, Jouanneau J, Chavrier P, Ladoux B,
  Buguin A, Silberzan P (2007) Collective migration of an epithelial monolayer
  in response to a model wound. Proceedings of the National Academy of Sciences
  104(41):{15}{988}--{15}{993}

\bibitem[{van Saarloos(2003)}]{van2003front}
van Saarloos W (2003) Front propagation into unstable states. Physics Reports
  386(2-6):29--222

\bibitem[{S{\'a}nchez-Gardu{\~n}o and Maini(1994)}]{sanchez1994existence}
S{\'a}nchez-Gardu{\~n}o F, Maini PK (1994) Existence and uniqueness of a sharp
  travelling wave in degenerate non-linear diffusion {F}isher-{K}pp equations.
  Journal of Mathematical Biology 33(2):163--192

\bibitem[{Sandstede(2002)}]{sandstede2002stability}
Sandstede B (2002) Stability of travelling waves. In: Handbook of {D}ynamical
  {S}ystems, vol~2, Elsevier, pp 983--1055

\bibitem[{Sattinger(1977)}]{sattinger1977weighted}
Sattinger D (1977) Weighted norms for the stability of traveling waves. Journal
  of Differential Equations 25(1):130--144

\bibitem[{Sewalt et~al.(2016)Sewalt, Harley, van Heijster, and
  Balasuriya}]{sewalt2016influences}
Sewalt L, Harley K, van Heijster P, Balasuriya S (2016) Influences of allee
  effects in the spreading of malignant tumours. Journal of Theoretical Biology
  394:77--92

\bibitem[{Sherratt(1998)}]{sherratt1998transition}
Sherratt JA (1998) On the transition from initial data to travelling waves in
  the {F}isher-{KPP} equation. {D}ynamics and {S}tability of {S}ystems
  13(2):167--174

\bibitem[{Sherratt et~al.(2014)Sherratt, Dagbovie, and
  Hilker}]{sherratt2014mathematical}
Sherratt JA, Dagbovie AS, Hilker FM (2014) A mathematical biologistâ€™s guide
  to absolute and convective instability. Bulletin of Mathematical Biology
  76(1):1--26

\bibitem[{Simpson and Landman(2007)}]{simpson2007nonmonotone}
Simpson MJ, Landman KA (2007) Nonmonotone chemotactic invasion:
  {H}igh-resolution simulations, phase plane analysis and new benchmark
  problems. Journal of Computational Physics 225(1):6--12

\bibitem[{Simpson et~al.(2010{\natexlab{a}})Simpson, Landman, and
  Hughes}]{simpson2010cell}
Simpson MJ, Landman KA, Hughes BD (2010{\natexlab{a}}) Cell invasion with
  proliferation mechanisms motivated by time-lapse data. Physica A: Statistical
  Mechanics and its Applications 389(18):3779--3790

\bibitem[{Simpson et~al.(2010{\natexlab{b}})Simpson, Landman, Hughes, and
  Fernando}]{simpson2010model}
Simpson MJ, Landman KA, Hughes BD, Fernando AE (2010{\natexlab{b}}) A model for
  mesoscale patterns in motile populations. Physica {A}: {S}tatistical
  {M}echanics and its {A}pplications 389(7):1412--1424

\bibitem[{Simpson et~al.(2010{\natexlab{c}})Simpson, Towne, McElwain, and
  Upton}]{simpson2010migration}
Simpson MJ, Towne C, McElwain DLS, Upton Z (2010{\natexlab{c}}) Migration of
  breast cancer cells: {U}nderstanding the roles of volume exclusion and
  cell-to-cell adhesion. Physical Review E 82(4):{041}{901}

\bibitem[{Simpson et~al.(2014)Simpson, Haridas, and
  McElwain}]{simpson2014pioneer}
Simpson MJ, Haridas P, McElwain DLS (2014) Do pioneer cells exist? PLOS ONE
  9(1):{e85}{488}

\bibitem[{Szmolyan and Wechselberger(2001)}]{szmolyan2001canards}
Szmolyan P, Wechselberger M (2001) Canards in $\mathbb{R}^{3}$. Journal of
  Differential Equations 177(2):419--453

\bibitem[{Wang et~al.(2019)Wang, Shi, and Wang}]{wang2018persistence}
Wang Y, Shi J, Wang J (2019) Persistence and extinction of population in
  reaction--diffusion--advection model with strong {A}llee effect growth.
  Journal of Mathematical Biology 78(7):2093--2140

\bibitem[{Wechselberger(2005)}]{wechselberger2005existence}
Wechselberger M (2005) Existence and bifurcation of canards in $\mathbb{R}^{3}$
  in the case of a folded node. SIAM Journal on Applied Dynamical Systems
  4(1):101--139

\bibitem[{Wechselberger(2012)}]{wechselberger2012propos}
Wechselberger M (2012) A propos de canards (apropos canards). Transactions of
  the American Mathematical Society 364(6):3289--3309

\bibitem[{Wechselberger and Pettet(2010)}]{wechselberger2010folds}
Wechselberger M, Pettet GJ (2010) Folds, canards and shocks in
  advection--reaction--diffusion models. Nonlinearity 23(8):1949--1969

\bibitem[{Weickert(1998)}]{weickert1998anisotropic}
Weickert J (1998) Anisotropic Diffusion in Image Processing, vol~1. Teubner
  Stuttgart

\bibitem[{Witelski(1994)}]{witelski1994asymptotic}
Witelski TP (1994) An asymptotic solution for traveling waves of a
  nonlinear-diffusion {F}isher's equation. Journal of Mathematical Biology
  33(1):1--16

\bibitem[{Witelski(1995)}]{witelski1995shocks}
Witelski TP (1995) Shocks in nonlinear diffusion. Applied {M}athematics
  {L}etters 8(5):27--32

\end{thebibliography}

\end{document}